\documentclass[reqno,12pt]{amsart}
\usepackage{amsmath}%
\usepackage{amsfonts}%
\usepackage{amsthm}
\usepackage{amssymb}
\usepackage{fullpage}
\usepackage{tikz}
\usepackage{stix}
\usepackage[mathscr]{eucal}
\usepackage{graphicx}
 \usepackage{tikz-cd}
 \usepackage{tikz,amsmath}
 \usepackage[utf8]{inputenc}
\usepackage[capitalise]{cleveref} 
\usepackage{mathtools}
\usepackage[utf8]{inputenc}

\newtheorem{theorem}{Theorem}[section]
\theoremstyle{plain}

\newtheorem{corollary}[theorem]{Corollary}

\newtheorem{definition}[theorem]{Definition}
\newtheorem{example}[theorem]{Example}
\newtheorem{lemma}[theorem]{Lemma}
\newtheorem{proposition}[theorem]{Proposition}
\newtheorem{remark}[theorem]{Remark}
\newcommand{\isom}{\cong}

\begin{document}
\title[C*-algebraic Grothendieck rings]{Remarks on computing the Grothendieck rings of C*-algebras}
\author{Elham Ghamari, Colin Ingalls, and Dan Z. Ku\v{c}erovsk\'{y} }
\address
{UNB-F \newline\indent Canada\qquad E3B 5A3}%
\email{eghamari@unb.ca,ColinIngalls@cunet.carleton.ca,dkucerov@unb.ca}%
\thanks{We thank NSERC for financial support. Authour's e-mail address: \texttt{dkucerov@unb.ca}}
\date{\fbox{\textbf{\textsc{\today}}}} 
\dedicatory{}
\begin{abstract}

\bigskip

In this paper, we present a captivating construction by Grothendieck, originally formulated for algebraic varieties, and adapt it to the realm of C*-algebras. Our main objective is to investigate the conditions under which this particular class of C*-algebras possesses a nontrivial Grothendieck ring. To achieve this, we explore the existence of nontrivial characters, which significantly enriches our understanding of these algebras. In particular, we conduct a detailed study of rings of C*-algebras over $\mathbb{C}$, $\mathbb{R}$, and $\mathbb{H}$.

 \end{abstract}\maketitle

\newcommand{\compose}{\circ}
\newcommand{\C}{\mathbb C}
\newcommand{\complex}{\C}
\newcommand{\N}{\mathbb N}
\newcommand{\Z}{\mathbb Z}
\newcommand{\R}{\mathbb R}
\newcommand{\M}[1]{M_{#1}(\C)} 
\newcommand{\tensor}{\otimes} 
\newcommand{\mintensor}{\otimes_{min}} 
\renewcommand{\isom}{\cong}
\renewcommand{\l}{\ell}
\hyphenation{mult-i-pli-cat-ive  co-mult-i-pli-cat-ive anti-mult-ipli-ca-tive co-anti-mult-ipli-ca-tive iso-mor-phism co-anti-iso-mor-phism anti-auto-mor-phism anti-auto-mor-phism auto-mor-phi-sms anti-auto-mor-phism  inter-sper-sed}

\section{Introduction}


C*-algebras are rings, sometimes nonunital, obeying certain axioms that ensure a very well-behaved representation theory upon Hilbert space. There are possibly as many as four interesting categories that can be built from C*-algebras. Firstly, where the objects are all C*-algebras and morphisms are all *-homomorphisms
between them. Secondly, where the objects are unital C*-algebras and the morphisms are unital *-homomorphisms.
Thirdly, where the objects are commutative C*-algebras and the morphisms are all *-
homomorphisms between them. Fourthly, where the objects are unital commutative C*-algebras and the morphisms are
the unital *-homomorphisms between them.  Moreover, there are some well-known features of the representation theory leading to subtle questions about norms on tensor products of C*-algebras, and thus to the subclass of nuclear C*-algebras. The question whether all separable nuclear C*-algebras satisfy the Universal Coefficient Theorem (UCT) remains one of the most important open problems in the structure and classification theory of such algebras. Although this problem is to our knowledge open, and has remained open for approximately fifty years, there exists a large supply of C*-algebras  that satisfy the UCT. First, there is a bootstrap process that allows constructing additional algebras satisfying the UCT starting from a few evident examples of algebras that do so. Secondly, there is a long-overlooked result of Tu \cite{Tu} showing that the large class of C*-algebras generated by Hausdorff, locally compact, second countable, and amenable groupoids satisfy the UCT. This has motivated research looking for C*-algebras that are of Tu's form for non-obvious reasons. Indeed, could Tu's class be unexpectedly large? This line of investigation has shown, for example, that separable, nuclear C*-algebras with Cartan subalgebras satisfy the UCT \cite{BarlakLi}.  There do not seem to be any known counterexamples to the UCT conjecture. On the other hand the K\"unneth theorem for tensor products which was first studied by Atiyah in the abelian case \cite{Atiyah} and Schochet in the general (nuclear) case \cite{Schochet1982} contains more classes of C*-algebras and it is still an open problem whether all separable nuclear C*-algebras satisfy the K\"unneth theorem for tensor product.

There is an interesting construction, introduced by Grothendieck in the context of algebraic varieties (see \cite{thesis} and the references there for more information on the case of algebraic varieties). We propose to adapt the construction to the case of C*-algebras.  Given a class of C*-algebras that is closed under tensor products, isomorphisms, formation of extensions, and finite direct sum, we can proceed as follows.  First consider a semigroup with C*-algebras drawn from some universe as the elements, and direct sum as the operation, then form the associated Grothendeick group. Now take the quotient by the equivalence relations
coming from isomorphism and from short exact sequences, meaning that whenever we have a short exact sequence
\[0 \rightarrow A \rightarrow C \rightarrow B\rightarrow 0,\]
we assert that $[C]=[A]+[B].$
Any C*-algebraic tensor product that is exact in the class under consideration will provide a ring product. As is well known, there is more than one tensor product on C*-algebras, and they need not be exact in general. The maximal tensor product, however, is always exact. The class of all C*-algebras with the maximal tensor product is thus an example of a class of C*-algebras that is closed under tensor products, isomorphisms, formation of extensions, and finite direct sums. Naturally, this is not the only class that can be considered. Indeed, Grothedieck's construction can be performed with  objects other than C*-algebras, provided that the categorical necessities are available. On the other hand, a purely categorical approach is limited in what it can tell us about the ring that has been constructed. Thus our aim here is to explore conditions for a class of C*-algebras to have a nontrivial Grothendieck ring, with emphasis on the existence of nontrivial characters. The interest in characters is partially motivated by the fact that just such a program has been quite successful in classical algebraic geometry. For example, the proof that the Grothedieck ring is not a domain rests on finding an appropriate character \cite{poonen}. Our aims here are modest, and we will be content to just show nontriviality of the Grothedieck ring through construction of an appropriate character. As a by-product of a result about ring ideals  (\cref{prop:ideal}) we show that some C*-algebras that are not obviously bootstrap class are in fact in that class.
\subsection{Organization of this paper}
The basic construction of Grothendieck rings for C*-algebras is introduced in section 2. We then develop in the next several sections a theory for constructing characters on Grothendieck rings, and our most advanced result in this direction is probably \cref{th:countably.generated.case}. We discuss ideals of Grothendieck rings in section 8. Section 9 focuses on examples, and in section 10 and the following, we turn to the case of real C*-algebras. In section 12 we study the particularly elegant case of real C*-algebras over the quaternions.

\section{The fundamental examples}
We now look more closely at two examples using C*-algebras.

Let us, then, for the moment  denote by $\mathcal C$ the proper class of all C*-algebras. We can avoid classes while reaching similar conclusions if we consider instead the set of all C*-algebras on a fixed Hilbert space.This forms a semigroup under direct sum, and there is an enveloping group, which may of course be trivial. We declare that whenever we have a short exact sequence
\[0 \rightarrow A \rightarrow C \rightarrow B\rightarrow 0,\]
we have $[C]=[A]+[B]$. Moreover, isomorphic C*-algebras by definition give the same semigroup element.  The maximal tensor product of C*-algebras is known to always be exact, and thus it respects the equivalence relation provided by short exact sequences. Therefore, we obtain a ring.


We now give a simple argument that shows this ring is trivial.  This argument will ultimately be used to support the perhaps surprising heuristic that large classes of C*-algebras tend to generate trivial or small Grothendieck rings. In C*-algebra theory we are permitted the construction of infinite direct products or sums.

\begin{theorem}
If a class of C*-algebras $\mathcal C$ is closed under the formation of either infinite products or infinite sums, then the Grothendieck ring of $\mathcal C$ is trivial.
\end{theorem}

\begin{proof}
 We recall that the direct product $\prod^\infty_1 A_i$ of countable many C*-algebras $A_i$ is a C*-algebra, given explicitly as the algebra of all bounded sequences $a_i$ with $a_i \in A_i$.
But then, given any C*-algebra A, the infinite direct product $D:=\prod^\infty_1 A,$ has the interesting property that $A\oplus D$ is isomorphic as a C*-algebra to $D.$  Since Grothendieck rings have isomorphism as one of their equivalence relations, if a C*-algebra $A$ and  $D$ constructed as above belongs to such a ring, it will then follow that $A=0$ in the ring. This is because the additive structure of such a ring is an abelian group. In other words, since $A\oplus D=D$ we have $[A\oplus D]=[D]$, and then $[A]+[D]=[D]$ which implies that every element in the ring is trivial. The case of direct sums is very similar.
\end{proof}

On the other hand, the familiar class of C*-algebras given by the AF algebras with finitely generated K-theory has a nontrivial Grothendieck ring. In the classical examples \cite{thesis}, it is usual to obtain information about the Grothedieck rings through their characters. This motivates us to look for a character, and to verify that the character takes nonzero values, thus implying that the Grothedieck ring is nonzero. First of all we define a possible character on the set of C*-algebras, and then we check that the definition does indeed pass to the Grothedieck ring.
We define a candidate character $\chi$ by employing the free part of the $K_0$-group of the AF algebra.This definition is compatible with isomorphism of AF-algebras, because isomorphic AF algebras have isomorphic K-theory groups. This definition  is also compatible with direct sums. Now, in the special case of AF-algebras, the $K_0$ functor is actually exact, and thus an exact sequence \[0 \rightarrow A \rightarrow C \rightarrow B\rightarrow 0\] gives an exact sequence of finitely generated abelian groups:
\[0 \rightarrow K_0 (A) \rightarrow K_0 (C) \rightarrow K_0 (B)\rightarrow 0.\]
Now tensoring with $\mathbb{Q}$ we get the following short exact sequence which splits:
\[0 \rightarrow K_0 (A) \otimes \mathbb{Q}  \rightarrow K_0 (C) \otimes \mathbb{Q} \rightarrow K_0 (B)\otimes \mathbb{Q} \rightarrow 0.\]
The above groups are a finite direct sum of copies of $\mathbb{Z} \otimes \mathbb{Q}$, and from the exactness of the sequence it follows that $\chi(C)=\chi(A)+\chi(B).$
 Multiplicativity is a deeper result. For AF algebras with finitely generated K-theory, the K\"unneth formula, to be further discussed,  implies that $K_0(A)\otimes K_0(B)=K_0(A\otimes B).$
This shows that there exists a nontrivial character of this Grothendieck ring, and thus the Grothendeick ring must be nontrivial.

Maybe we should point out that there is something like a simpler equivalent form of the above Grothedieck ring,  for essentially categorical reasons. The AF algebras are, as is well-known, classified by K-theory \cite{elliottAF}. As was discussed above, in the case of AF algebras, the $K_0$ functor is compatible with the equivalence relations that define the Grothedieck ring.  Thus the $K_0$ functor gives a well-defined map of Grothedieck rings, and therefore, the above Grothedieck ring is  isomorphic to the more elementary Grothedieck ring constructed from the finitely generated abelian dimension groups.
On the other hand, since  the rings are isomorphic, it is a question of viewpoint which picture one prefers.

In our two examples, the class of C*-algebras used had in both cases some very strong stability properties. This is more or less a tautology for the first example, where we considered the class of all C*-algebras. The second example needs some discussion. First of all, an ideal or a quotient of an AF algebra is again an AF algebra, with finitely generated K-theory if we began with an algebra having finitely generated K-theory. It is also true that if the ideal and quotient algebras in an extension are AF, then the extension algebra is AF.  Moreover, an algebra is AF exactly when a finite set of elements  can be approximated in norm by a finite-dimensional subalgebra. This characterization is preserved under isomorphism, thus, an algebra isomorphic to an AF algebra is itself AF. These many properties are more than enough to allow the needed equivalence relations, thus we can form a Grothedieck ring. This is for the following reasons: Firstly, AF algebras are nuclear, thus there is only one tensor product, up to isomorphism. Indeed, the tensor product of AF algebras is AF. If we have a short exact sequence as above, with middle element $C$ in the class, then the ideal $A$ and the quotient $B$ are again in the class. For completeness, we recall that this class of C*-algebras satisfy the K\"unneth formula for tensor products in $K$-theory \cite{Schochet1982}, which means that there is a short exact sequence:

  \begin{multline*}0 \rightarrow K_0(A) \otimes K_0(A) \oplus K_1(A) \otimes K_1(A) \rightarrow K_0(A \otimes A) \rightarrow\\ \mathrm{Tor}(K_0(A),K_1(A)) \oplus \mathrm{Tor}(K_1(A),K_0(A)) \rightarrow 0.\end{multline*}

 It follows that from the fact that for an AF algebra, $K_1$ is trivial, that $$K_0(A) \otimes K_0(A) \rightarrow K_0(A \otimes A)$$ is injective and surjective, thus an isomorphism $$t:K_0(A) \otimes K_0(A) \rightarrow K_0(A \otimes A).$$ In other words, the $K_0$ functor is multiplicative on AF algebras.

\section{On the UCT and the KTP class}
As defined by Rosenberg and Schochet in \cite{SchochetRosenberg}, a separable C*-algebra A is said to satisfy the
universal coefficient theorem (UCT) if for every separable C*-algebra B, the following sequence is exact
\[0 \rightarrow {\rm Ext}(K_* (A) , K_*(B)) \rightarrow KK_* (A,B) \rightarrow\\ {\rm Hom}(K_*(A),K_* (B))\rightarrow 0.\]
where the right hand map is the natural one and the left hand map is the inverse of a map
that is always defined.

On the other hand, Kirchberg has shown that every separable, nuclear C*-algebra is KK-equivalent to a Kirchberg algebra, see \cite[Theorem I]{KirchToAppear}. Recall that Kirchberg algebras are separable simple nuclear and purely infinite C*-algebras, and a simple C*-algebra is purely infinite if every nonzero hereditary sub-C*-algebra contains an infinite projection, where an infinite projection is a projection which is Murray-von Neumann equivalent to a proper sub-projection of itself.
\begin{theorem}(Kirchberg)
The UCT holds for all separable nuclear C*-algebras if and only if it holds for all unital Kirchberg algebras.
\end{theorem}
 Thus the famous UCT conjecture becomes equivalent to the conjecture that every Kirchberg algebra is KK-equivalent to a commutative C*-algebra, or to the conjecture that every Kirchberg algebra satisfies the UCT.

The UCT is known to hold for a large class of C*-algebras. We therefore define the term UCT class to mean the class of C*-algebras $A$ for which the universal coefficient theorem (UCT) holds.  The UCT class is closed under KK-equivalence, tensor product, inductive limits, two out of three in a short exact sequence and under crossed products by $\mathbb{Z}$ or $\mathbb{R}$ \cite{SchochetRosenberg}.

Rosenberg and Schochet showed that a C*-algebra is in the UCT class if and only if it is KK-equivalent to a commutative C*-algebra. In the absence of a more effective characterization of the UCT class, it is conventional to define
a subclass of the  UCT class called the bootstrap class $N$. The bootstrap class $N$ is smaller than the UCT class, but it has a practical definition that is sometimes feasible to check: it is generated by the C*-algebra $\complex$ and the equivalence relations given by stable isomorphism, closure under inductive limits, closure under the two out of three property, and closure under crossed product by $\mathbb{Z}.$
We note that the terms bootstrap class and UCT class are not always used in the same sense by all authours.
 Let us point out that, with our terminology, the bootstrap class $N$ appears to be properly contained in the UCT class. As previously mentioned,  work of Tu \cite{Tu}, see also \cite{Oyono}, in fact showed that a large class of groupoid algebras that do not seem to be in the bootstrap class are definitely in the UCT class. See also \cref{prop:ideal}.

The main results about UCT algebras that we will  use are the following K\"unneth-type  formulas for tensor products in $K$-theory. The first one just summarizes the above discussion:

  \begin{theorem} \cite{SchochetRosenberg} Let $A$ be in the UCT class. Then there is a short exact sequence:
\[0 \rightarrow {\rm Ext}(K_* (A) , K_*(B)) \rightarrow KK_* (A,B) \rightarrow\\ {\rm Hom}(K_*(A),K_* (B))\rightarrow 0.\]
	\end{theorem}
In the above sequence, $K_*$ represents a graded direct sum of the usual complex $K$-theory groups $K_0$ and $K_1.$ The first map in the exact sequence  has degree 0 and the second map has degree 1. To the best knowledge of the authors, there is no example where the above sequence fails if we choose the tensor product to be the maximal tensor product of C*-algebras, and it seems possible that in this case such a sequence may always hold.
There exist related sequences that can be proved similarly to the UCT:
	
  \begin{theorem} [\protect{K\"unneth theorem \cite{SchochetRosenberg}}]   Let $A$ be in the UCT class, and suppose that one of $A$ \emph{or} $B$ has finitely generated K-theory. Then there is a short exact sequence:
  \begin{align*}
    0 \rightarrow K^{*}(A) \otimes K_{*}(B) \xrightarrow\alpha KK_{*}(A,B) \xrightarrow\beta  \mathrm{Tor}_1 (K^{*}(A),K_*(B))  \rightarrow 0
  \end{align*}
  which is natural in each variable. The map $\alpha$ has degree $0$ and the map $\beta$ has degree $1$.
	\end{theorem}
In the above theorem, some finite generation hypothesis seems necessary, because of Elliott's counterexample, in which both $A$ and $B$ are taken to be UHF algebras with dimension group (\textit{i.e.} the first K-theory group) given by the additive group of the rational numbers $\mathbb{Q}.$ The next theorem, however, does not need a finite generation hypothesis, and will be especially important to our work:
 \begin{theorem}[\protect{KTP theorem\cite{SchochetRosenberg}}]
 Let $A$ \emph{or} $B$ be in the UCT class. Then there is a short exact sequence:
 \begin{align*}
    0 \rightarrow K_{*}(A) \otimes K_{*}(B) \xrightarrow\alpha K_{*}(A\otimes B) \xrightarrow\beta  \mathrm{Tor}_1 (K_{*}(A),K_*(B))  \rightarrow 0.
  \end{align*}

 \end{theorem}

 The map $\alpha$ in the KTP has degree $0$ and the map $\beta$ has degree $1$, so we rewrite the KTP  theorem in terms of the following sequences, as we use them frequently in our proofs.

 \begin{theorem} \label{KTP} Let $A$ be in the KTP class. Then there are short exact sequences:
  \begin{multline*} 0 \rightarrow K_0(A) \otimes K_0(B) \oplus K_1(A) \otimes K_1(B) \rightarrow K_{0}(A \tensor B) \rightarrow \\ \mathrm{Tor}(K_0(A),K_1(B)) \oplus \mathrm{Tor}(K_1(A),K_0(B))
  \rightarrow 0 \end{multline*}
  and
   \begin{multline*} 0 \rightarrow K_0(A) \otimes K_1(B) \oplus K_1(A) \otimes K_0(B) \rightarrow K_{1}(A \tensor B) \rightarrow \\ \mathrm{Tor}(K_0(A),K_0(B)) \oplus \mathrm{Tor}(K_1(A),K_1(B))
  \rightarrow 0. \end{multline*}\label{th:KPT}
	\end{theorem} 
	
On the other hand, we say that $A$ satisfies the K\"unneth formula if and only if

\begin{align*}
\alpha: K_{*}(A) \otimes K_{*}(B) \rightarrow K_{*}(A\otimes B)
\end{align*}

is an isomorphism for any C*-algebra $B$ with free abelian K-groups $K_{*}(B)$ \cite{Oyono}. This is equivalent to say that $A$ satisfies the K\"unneth formula if there is a short exact sequence
	
 \begin{align*}
    0 \rightarrow K_{*}(A) \otimes K_{*}(B) \xrightarrow\alpha K_{*}(A\otimes B) \xrightarrow\beta  \mathrm{Tor}_1 (K_{*}(A),K_*(B))  \rightarrow 0,
  \end{align*}

as already mentioned, we call the class of all C*-algebras that satisfy the K\"unneth formula for tensor product, the KTP class. Every C*-algebra which in the UCT class is KK-equivalent to a commutative C*-algebra, and Atiyah \cite{Atiyah} showed that every commutative C*-algebra is in the KTP class. Therefore, the KTP class contains the UCT class.  Indeed, there are some examples which satisfy the K\"unneth formula for tensor product but they are not in UCT class.
 It follows from the "going down functor" machinery of \cite{Oyono} that if $G$ is any group that satisfies the Baum-Connes conjecture with coefficients then $C_{r}^{*}(G)$ satisfies K\"unneth formula. On the other hand, if $G$ is an infinite hyperbolic property $(T)$ group then $C_{r}^{*}(G)$ does not satisfy the UCT \cite{Skandalis}.

  Since the bootstrap category $N$ is in the UCT class it is in the KTP class too, so the KTP class contains a great many of the common examples of C*-algebras; and also it is  closed under properties that are required in this paper such as tensor product, inductive limit and two out of three in a short exact sequence  \cite{Oyono}.

	We make a minor contribution towards the UCT conjecture, in the following proposition. The proposition also shows the existence of a  ring ideal that is \emph{not} obviously generated by an idempotent, in a certain Grothendieck ring.
				\begin{proposition} Let $A$ be a separable nuclear C*-algebra, let $B$ be in the subset of UCT class algebras with $K_*=\{e\}.$ Then $A\tensor B$ is separable nuclear, in the UCT class and has $K_*=\{e\}.$ \label{prop:ideal}\end{proposition}
				\begin{proof}
				Recall that the UCT holds if just one algebra is in the UCT class and has finitely generated K-theory. Clearly, the algebra $B$   has  these properties.  Thus,
				  \begin{multline*}0 \rightarrow K_0(A) \otimes K_0(B) \oplus K_1(A) \otimes K_1(B) \rightarrow K_0(A \otimes B) \rightarrow\\ \mathrm{Tor}(K_0(A),K_1(B)) \oplus \mathrm{Tor}(K_1(A),K_0(B)) \rightarrow 0\end{multline*}
					and from this it follows that the $K$-theory of $A\otimes B$ is trivial since $B$ has trivial $K$-theory. Thus all the difficulty lies in showing that $A\tensor B$ has the UCT property, and for this it is sufficient to show $KK$-equivalence to an algebra in the UCT class.   Since $B$ is separable and nuclear, it is $KK$-equivalent to a purely infinite simple nuclear algebra, $B'.$ Since $B'$ is a Kirchberg algebra and has trivial K-theory, by Kirchberg's classification result, it is therefore isomorphic to $O_2.$ Stabilizing, we thus have that $A\tensor B$ is $KK$-equivalent to $A\tensor K \tensor O_2.$  Now we recall \cite{IPZH} the deep result that if $A$ is separable and nuclear, then  $A\tensor K \tensor O_2$ is isomorphic to the crossed product $D\rtimes_\beta Z$ where $D$ is an $AH_0$ class algebra and $\beta$ is an automorphism. Thus since $D$ is an inductive limit of finite direct sums of algebras of the form $C_0(X)\tensor M_n$ it is in the UCT class, and crossed products by $Z$ preserve the UCT class.
					Thus we are done.
\end{proof}

\section{Additive maps and cyclic sequences}
In this section we define two different additive maps on a collection of complex C*-algebras and we check their additivity at  the Grothendick ring level. We begin by defining an additive function on abelian groups.
\begin{definition} \label{functionS}
Let $S$ be a real valued function on abelian groups $G_{i}$. We say $S$ is additive if a short exact sequence
$$0\longrightarrow G_{1} \longrightarrow G_{2} \longrightarrow G_{3} \longrightarrow 0 $$
implies $S(G_{2} )=S(G_{1} )+S(G_{3} )$.
\end{definition}

\begin{lemma} \label{function}
Consider the cyclic exact sequence
\[
\begin{tikzcd}
G_{1}  \ar[r, "\varphi_{1}"]& G_{2}   \ar[r,"\varphi_{2}"] & G_{3} \ar[d,"\varphi_{3}"] \\
G_{6} \ar[u,"\varphi_{6}"] & G_{5}  \ar[l,"\varphi_{5}"] & G_{4} \ar[l,"\varphi_{4}"]
\end{tikzcd}
\]
and let $S$ be additive, then
\begin{align}
(S(G_{5})-S(G_{2}))=(S(G_{4})-S(G_{1}))+(S(G_{6})-S(G_{3})).
\end{align}
\begin{proof}
From the cyclic exact sequence we have the following short exact sequences

$$0\longrightarrow  \ker\varphi_{1} \longrightarrow G_{1} \longrightarrow \text{im } \varphi_{1} \longrightarrow 0 $$
	and
	$$0\longrightarrow  \ker\varphi_{2} \longrightarrow G_{2} \longrightarrow \text{im }\varphi_{2} \longrightarrow 0 $$
and so on. The property of $S$ on a short exact sequence implies
$$
S(G_{1})=S(\ker\varphi_{1})+S(\text{im }\varphi_{1})
$$
but $\text{im }\varphi_{1} =\ker\varphi_{2}$, therefore
$$S(G_{1})=S(\ker\varphi_{1})+S(\ker\varphi_{2}), S(G_{2})=S(\ker\varphi_{2})+S(\ker\varphi_{3}), ...$$
which implies that
$$
S(G_{1})-S(G_{2})+S(G_{3})-S(G_{4})+S(G_{5})-S(G_{6})=0.
$$
\end{proof}
then
\begin{align}
(S(G_{5})-S(G_{2}))=(S(G_{4})-S(G_{1}))+(S(G_{6})-S(G_{3})).
\end{align}
\end{lemma}
\begin{lemma} \label{Folding}
Consider the cyclic exact sequence
\[
\begin{tikzcd}
G_{1}  \ar[r, "\varphi_{1}"]& G_{2}   \ar[r,"\varphi_{2}"] & G_{3} \ar[d,"\varphi_{3}"] \\
G_{6} \ar[u,"\varphi_{6}"] & G_{5}  \ar[l,"\varphi_{5}"] & G_{4} \ar[l,"\varphi_{4}"]
\end{tikzcd}
\]
then the following sequence

%

  \begin{center}\begin{tikzpicture}[auto]
    \node (M1) at  (0,0) [left] {$\begin{smallmatrix}G_1\\\oplus\\G_4\end{smallmatrix}$};
    \node (M2) at (1,0) [] {$\begin{smallmatrix}G_2\\\oplus\\G_5\end{smallmatrix}$};
     \node (M3) at (2,0) [right] {$\begin{smallmatrix}G_3\\\oplus\\G_6\end{smallmatrix}$};
    \draw[->] (M1) to node [swap] {$\begin{smallmatrix}\varphi_1\\\oplus\\\varphi_4\end{smallmatrix}$} (M2);
     \draw[->] (M2) to node [swap] {$\begin{smallmatrix}\varphi_2\\\oplus\\\varphi_5\end{smallmatrix}$} (M3);
    \draw[->,bend right=135] (M3) to node [swap] {$f \circ \sigma$} (M1);
  \end{tikzpicture}\end{center}

   is also exact, where f is $(\varphi_3 \oplus \varphi_6)$, and $\sigma$ is the flip interchanging the direct summands.
\end{lemma}
\begin{corollary}
Let $\bar{S}(G_{n} \oplus G_{n+3}):=S(G_{n+3})-S(G_{n})$, then $\bar{S}$ is additive on the sequence in Lemma\eqref {Folding}, in the sense that
\begin{align*}
\bar{S}(G_{2} \oplus G_{5})=\bar{S}(G_{1} \oplus G_{4})+\bar{S}(G_{3} \oplus G_{6})
\end{align*}
\end{corollary}

\section{F\o lner sequences and amenability}
Next we show that there are maps on the Grothendieck rings of C*-algebras which are additive but are not multiplicative in general. In order to define these kinds of maps we take advantage of F\o lner sequences. Throughout this section we suppose that the groups are discrete.

Let $G$ be a group, for $f: G \rightarrow \mathbb{C}$ and each $g \in G,$ we define a left translation action by $(g.f)(h)=f(g^{-1}h)$ for every $h \in G$.

\begin{definition}
A group $G$ is amenable if it admits a left-invariant mean: that is , there is a state $\mu$ on $\l^{\infty}(G)$ such that $\mu(g^{-1}.f)=\mu(f)$, for every $f \in \l^{\infty}(G)$ and $g \in G$.
\end{definition}
\begin{example}
All finite groups are amenable, the left-invariant mean for a finite group $G$ is defined by $\mu(f)=|G|^{-1} \Sigma_{g \in G} f(g)$, for every $f \in \l^{\infty}(G)$.
\end{example}

\begin{definition}
For a countable group $G$ a F\o lner sequence is a sequence $\{ F_{n} \}$ of non-empty finite subset of $G$ such that
\begin{align*}
\frac{|gF_{n} \Delta F_n|}{|F_{n}|} \rightarrow 0
\end{align*}
for every $g \in G$, where $\Delta$ is the symmetric difference operator, i.e., $A \Delta B$ is the union of both relative complements of $A$ and $B$, and $|A|$ is the cardinality of a set $A$.
\end{definition}
The following lemma is well-known.
\begin{lemma}\label{amen&fol}
A discrete group $G$ has a F\o lner sequence if and only if it is amenable.
\end{lemma}
Certainly the group $\Z$ is amenable.
For example, we could choose the F\o lner sequence $\mathscr{F}_n=\{-n,\cdots ,n\}.$
Since every countable abelian group is the direct limit of finitely generated abelian groups, and these  are amenable, every countable abelian group is amenable.

\begin{definition}\label{linear map}
Let $A$ be a C*-algebra whose $K$-groups are finitely generated abelian groups.  Let $\mathscr{F}_n$ be a F\o lner sequence of the group $\mathbb{Z}$. Define the map $\mathrm{R}(A)$ on the set of these C*-algebras as follows
\begin{align*}
\mathrm{R} (A):= \lim_{n \rightarrow \infty} \frac{\ln(|G_{n}^{A}|) -\ln(|H_{n}^{A}|) }{\ln(|\mathscr{F}_n|) },
\end{align*}
where $G_{n}^{A}=K_0(A)\otimes \mathscr{F}_n$ and $H_{n}^{A}=K_1(A)\otimes \mathscr{F}_n$.
\end{definition}
As we will show next, the above map can be described in simpler terms. However, the above definition generalizes easily to the countably generated case, whereas the simplified description does not.
The rank of a finitely generated abelian group is the minimal number of copies of $\mathbb{Z}$ in the free part of the group. We find that $R(A)$ is given in terms of the ranks of the $K$-theory groups.\par
First a lemma:
\begin{lemma}
If $K$ is a finitely generated abelian group, then
\[ \lim_{n\to\infty} \frac{\ln |K\otimes \mathscr{F}_n|}{\ln |\mathscr{F}_n|}=M \]
where $M$ is the rank of $K$, i.e., the number of generators of the free part of $K$.
\label{lem:cardinality.of.F.1}
\end{lemma}
\begin{proof}
By the structure theorem for finitely generated abelian groups, $K=\Z^M \oplus T,$ where $T$ is a finitely generated torsion group, and thus has finitely many elements.
For sufficiently large $n$, we notice that $K\otimes \mathscr{F}_n$ will equal, as a set, $[-n,n]^M \oplus T,$ and so the cardinality  $|K\otimes \mathscr{F}_n|$ will be a multiple of $(2n+1)^M.$
But from this it follows that \( \lim_{n\to\infty} \frac{\ln |K\otimes \mathscr{F}_n|}{ \ln |\mathscr{F}_n|}=M, \) as claimed.
\end{proof}
From this lemma it follows that:
\begin{proposition} If $K_0 (A) =\mathbb{Z}^n \oplus \mathbb{Z}_{q_1} \oplus \cdots \oplus \mathbb{Z}_{q_t}$ and
                                      $K_1 (A) =\mathbb{Z}^m \oplus \mathbb{Z}_{q'_1} \oplus \cdots \oplus \mathbb{Z}_{q'_{t'}}$
then $R(A)=n-m.$
\end{proposition}

Now we claim that the map $\mathrm{R}$ is additive on the Grothendick rings with respect to the operation $+$. In order to show our claim, we consider a short exact sequence of C*-algebras
\begin{align*}
0 \rightarrow A \rightarrow C \rightarrow B\rightarrow 0,
\end{align*}
and  we will show that
$\mathrm{R} (C)=\mathrm{R} (A) + \mathrm{R} (B)$.
\section*{Additive and multiplicative maps: the rank characters}
In this section we further develop the characters introduced in the last section. These will be called the rank characters and we show that on their domain of definition they are not only additive but also multiplicative.
This will first of all require assuming finitely generated K-theory, but then we will show how to modify  this condition.
\begin{definition}
Consider a collection of complex C*-algebras. Define a map $\chi_{1}(A):=m-n$ where $m$ is the number of free generators of the finitely generated abelian group $K_1(A)$ and $n$ is the number of free generators of the finitely generated abelian group $K_0(A)$, for every  C*-algebra $A$ in the collection.
\end{definition}

Consider the following well-known lemma, which we will utilize to demonstrate the additivity of the map $\chi_{1}$ defined earlier.

\begin{lemma} \label{split}
Consider a short exact sequence of abelian groups
$$0\longrightarrow G_{1} \longrightarrow G_{2} \longrightarrow G_{3} \longrightarrow 0, $$
 tensoring by $\mathbb{Q}$ we get the short exact sequence of $\mathbb{Q}$-modules
$$0\longrightarrow G_{1} \otimes \mathbb{Q}  \longrightarrow G_{2}\otimes \mathbb{Q}  \longrightarrow G_{3}\otimes \mathbb{Q} \longrightarrow 0 $$
which splits, in other words
$$
G_{2}\otimes \mathbb{Q}=(G_{1} \otimes \mathbb{Q}) \oplus (G_{3}\otimes \mathbb{Q}).
$$
\end{lemma}

\begin{corollary} \label{additive-chi1}
Consider a collection of complex C*-algebras with finitely generated K-theory, then $\chi_{1}$ as defined above is additive.
\end{corollary}

\begin{proof}
 We recall that the additivity means that if
\begin{align} \label{ses-1}
0 \longrightarrow A \longrightarrow C \longrightarrow B \longrightarrow 0
\end{align}
  is a short exact sequence of complex C*-algebras, then $\chi_{1}(C)=\chi_{1}(A)+\chi_{1}(B)$. Now, focusing on this sequence, we observe that it gives rise to a cyclic exact sequence

\[
\begin{tikzcd}
K_{0}(A)  \ar[r, "\varphi_{1}"]& K_{0}(C)   \ar[r,"\varphi_{2}"] & K_{0}(B) \ar[d,"\varphi_{3}"] \\
K_{1}(B) \ar[u,"\varphi_{0}"] & K_{1}(C)  \ar[l,"\varphi_{5}"] & K_{1}(A)  \ar[l,"\varphi_{4}"]
\end{tikzcd}
\]

tensor this cyclic exact sequence by $\mathbb{Q}$ we get the cyclic exact sequence

\[
\begin{tikzcd}
K_{0}(A) \otimes \mathbb{Q}  \ar[r, "\varphi_{1} \otimes id"]& K_{0}(C)  \otimes \mathbb{Q}  \ar[r,"\varphi_{2} \otimes id"] & K_{0}(B) \otimes \mathbb{Q}  \ar[d,"\varphi_{3} \otimes id"] \\
K_{1}(B) \otimes \mathbb{Q}  \ar[u,"\varphi_{0}\otimes id"] & K_{1}(C) \otimes \mathbb{Q}  \ar[l,"\varphi_{5}\otimes id"] & K_{1}(A) \otimes \mathbb{Q}  \ar[l,"\varphi_{4}\otimes id"]
\end{tikzcd}
\]

Considering the cyclic exact sequence, we can obtain some related short exact sequences. These sequences can be written in the form:
$$0\longrightarrow  \ker(\varphi_{1} \otimes id) \longrightarrow K_{0}(A) \otimes \mathbb{Q}  \longrightarrow \text{im }(\varphi_{1}\otimes id) \longrightarrow 0 $$
	and
	$$0\longrightarrow  \ker(\varphi_{2}\otimes id) \longrightarrow K_{0}(C) \otimes \mathbb{Q}    \longrightarrow \text{im }\varphi_{2} \otimes id)\longrightarrow 0 $$
and so on, now apply Lemma \eqref{split}, each of these sequences split and then Lemma \eqref{function} shows that $$\chi_{1}(C)=\chi_{1}(B)+\chi_{1}(A).$$
\end{proof}

Therefore we showed that there exist linear maps on the Grothendick rings of C*-algebras, however, these maps are not multiplicative in general which means that the map $\mathrm{R} $ is not a character for a Grothendick ring of C*-algebras in general. Indeed, for checking the multiplicativity, we need to apply the KTP sequence which implies that the map $\alpha :K_{*}(A) \otimes K_{*}(B) \rightarrow K_{*}(A\otimes B) $ is not isomorphism, in general.
We can also define at least a linear map using the $p$-localization functor that picks out the $p$-power part of an abelian group.
This map is defined by $\chi_p (A):= n-m$ where $n$ is the base-$p$ logarithm of the cardinality of $K_0(A)\otimes \mathbb{Z}_p$ and $m$ is the base-$p$ logarithm of the cardinality of $K_1(A)\otimes \mathbb{Z}_p.$
Then, as localization is an exact functor we have a cyclic exact sequence of $p$-power groups
$$\begin{array}{rcccccl}
K_0(B)_p&\longrightarrow&K_0(C)_p&\longrightarrow&K_0(A)_p\\
\uparrow && &&\downarrow\\
K_1(A)_p&\longleftarrow&K_1(C)_p&\longleftarrow&K_1(B)_p\\\end{array}
$$		
and considering kernels and co-kernels shows that $\chi_p(C)=\chi_p(B)+\chi_p(A),$ as expected. However,  this map does not appear to be multiplicative in general.

In the above constructions, we assumed finitely generated $K$-theory. There exists however a simple example of a C*-algebra that does not have finitely generated $K$-theory:

The so-called universal complex UHF-algebra $U$ is a complex UHF-algebra for which $K_{0}(U)=\mathbb{Q}$ and $K_{1}(U)=0$. In fact, from the known classification theorem (see \cref{glimm}) for UHF algebras, it follows that this nuclear C*-algebra is a tensor idempotent. It therefore seems desirable to extend the above theory to the case of countably generated abelian groups. We are going to introduce a ``renormalized'' rank function that is nontrivial exactly when the usual rank is infinite.
Recall from \cite{TerenceTao} that:
\begin{lemma} For the free abelian group $\mathbb{Z}^\infty$ on countably many generators,
the sets $\mathcal G_n := \{-N_n,\cdots,N_n\}^n$ will form a F\o lner sequence if the increasing sequence $N_n$ grows sufficiently rapidly in $n.$ \end{lemma}

In other words, the countable free abelian group $\mathbb{Z}^\infty,$ an infinite direct sum, is amenable. Amenability passes to quotients, so then any countable abelian group is amenable.
\begin{remark}
Given a countably generated abelian group, $G,$ the fact that any such group is a quotient of a countable free abelian group means we always have a short exact sequence:
 $$ 0\to K \to  \mathbb{Z}^\infty \to G \to 0.$$
 Since the middle term $\mathbb{Z}^\infty$ is amenable, there is a Haar measure $m$ on $\mathbb{Z}^\infty$ and we can use it to measure the relative size of $K$ in $\mathbb{Z}^\infty$ and this indirectly measures the size of $G$.
 For example, if we define a renormalized rank function to be $r(G):=-\ln(m(K)),$ this will then be a non-negative real number that measures the size of quotient groups between $\{0\}$ and  $\mathbb{Z}^\infty.$ We would then obtain expected properties such as
 $r(G_1 \oplus G_2)= r(G_1) + r(G_2).$ However, it seems technically better to instead define our  ``renormalized'' rank function using F\o lner sequences instead of the measures obtained from amenability.
\end{remark}

We should mention that the finite generation hypothesis is conventional but is not really necessary.
Finite generation was used when we said that fixing a F\o lner sequence $\mathscr{F}_n$ for $\Z,$ we obtain a F\o lner sequence for any finitely generated abelian group $G,$ by considering $\mathscr{F}_n\otimes G.$ However, we could drop finite generation by considering more and more generators as $n$ increases. The following lemma indicates how this can be done:

\begin{lemma}\label{lem:4.18} If $G$ is a countably generated abelian group, with generators $(g_i),$ then from the above F\o lner sequence $\mathcal G_n$ for $\mathbb{Z}^\infty$ we can construct a F\o lner sequence for $G.$  \end{lemma}
\begin{proof}

To say that the $(g_i)$ are generators means that we consider $G$ to be given by $\sum z_i g_i$ where the $g_i$ are the generators and the $z_i$ are (signed) integer coefficients. If we are given a F\o lner sequence of the form $\{-N_n\cdots,N_n\}^n$ for $\mathbb{Z}^\infty,$ then in view of the fact that the given group $G$ is a quotient of $\mathbb{Z}^\infty,$ we can simply apply the canonical quotient map $\pi: \mathbb{Z}^\infty \to G$ to the given F\o lner sequence. This means that the sequence of set products denoted $\prod^{n}_{k=1} \{-N_n g_k, \cdots,N_n g_k\}$ is  a F\o lner sequence for $G.$ This is the desired construction, and gives us a reasonably explicit formula for a F\o lner sequence of $G.$
\end{proof}

Define the notation $\mathcal G_n \cap G$ for a F\o lner sequence of the above form. The F\o lner sequence $\mathcal G_n$ provides qualitatively similar but distinct behaviour from the previously considered F\o lner sequence $\mathcal F_n$.
This can be seen by comparing the above Lemma \eqref{lem:cardinality.of.F.1} with Lemma \eqref{lemma:TorsionDoesntMatter} below.

\begin{definition}
Let $G$ be a countably generated abelian group, then we define a map $q$ on $G$ as follows:
\begin{align*}
q(G):= \lim_{n\to\infty}\frac{\ln | \mathcal G_n \cap G |}{\ln | \mathcal G_n |}.
\end{align*}
\end{definition}

\begin{lemma} \label{lemma: q is additive}
Let $G$ be a countably generated abelian group, then $q$ is additive on a short exact sequence.
\end{lemma}

\begin{proof}
Let
\[ 0 \rightarrow H \rightarrow G \rightarrow K \rightarrow 0 \]
be a short exact sequence of countably generated abelian groups. Then we have
\[ 0 \rightarrow H \cap \mathcal G_n  \rightarrow G \cap \mathcal G_n \rightarrow K \cap \mathcal G_n \rightarrow 0. \]
Now taking the cardinality we have
\begin{align*}
|G \cap \mathcal G_n |= |H \cap \mathcal G_n | |K \cap \mathcal G_n |,
\end{align*}
then comparing with the definition of  $q$ we have $q(G)= q(H)+q(K)$.
\end{proof}

\begin{lemma} \label{lemma:TorsionDoesntMatter}
 If a countably generated abelian group $G$ has only torsion elements, then $q(G)=0$.
 \end{lemma}
 \begin{proof}
  This is because for a torsion element $g_k$, the ratio $\frac{|\{-N_n g_k, \cdots,N_n g_k\}|}{|\{-N_n , \cdots,N_n \}|}$ is less than one if $N_n$ exceeds the order of $g_k,$ and ultimately goes to zero. \end{proof}

Just as before, we can define a character using this F\o lner sequence.
\begin{definition}\label{linear map}
Let $A$ be a C*-algebra whose $K$-groups are countably generated abelian groups. Let $\mathcal G_n$ be a F\o lner sequence for the group $\mathbb{Z}^\infty$. Define the map $Q(A)$ on the set of these C*-algebras as follows
\begin{align}
Q (A):= q(K_{0}(A)) - q(K_{1}(A)) .
\end{align} \label{infinite rank}
\end{definition}
The proof that the above map is multiplicative, as well as additive; and thus is a character on a suitable domain, is deferred to Theorem \ref{th:countably.generated.case}.

\begin{example}
Assume that $K_{0}(A)=\mathbb{Z}^{\infty}$ and $K_{1}(A)=0$. By the definition of $q$, we have $Q (A)=\lim_{n \rightarrow \infty} \frac{\ln| \mathcal G_n|}{\ln|\mathcal G_n|}=1$.
\end{example}

\begin{example}
Suppose that $K_{0}(A)=\mathbb{Z}^{n}$ and $K_{1}(A)=0$. In this scenario, it follows that $Q (A)=0$.
\end{example}

\section{Defining a character by tensoring by C*-algebras}

The second approach that we will propose to define a character on the Grothendieck rings, following a e-mail conversation with C. Schochet, is to use the universal UHF-algebra. In the following lemma we show that complex UHF-algebras and therefore the universal UHF-algebra are in the KTP class.

\begin{lemma}
Let $U$ be a complex UHF-algebra. Then $U$ is in the KTP class.
\end{lemma}
\begin{proof}
Since $U$ is a complex UHF-algebra, then it is the inductive limit of the sequence
\begin{align}
M_{k_{1}}(\mathbb{C}) \xrightarrow{\varphi_{1}} M_{k_{2}}(\mathbb{C}) \xrightarrow{\varphi_{2}}  M_{k_{3}}(\mathbb{C}) \xrightarrow{\varphi_{3}} ...
\end{align}
but $M_{k_{i}}(\mathbb{C})$ is in the KTP class for every $i$, and KTP  is closed under the inductive limit, hence $U$ is in the KTP class.
\end{proof}

\begin{lemma}\label{UHF}
Let $A$ be any complex C*-algebra and $U$ be the universal UHF-algebra, then $K_{*}(A \otimes U)$ is a $\mathbb{Q}$-module.
\end{lemma}

\begin{proof}
Suppose that $U$ is the universal complex UHF-algebra, since every complex UHF-algebra belongs to the KTP class, therefore there exists a short exact sequences
\begin{multline*} 0 \rightarrow K_0(A) \otimes K_0(U) \oplus K_1(A) \otimes K_1(U) \rightarrow K_{0}(A \tensor U) \rightarrow \\ \mathrm{Tor}(K_0(A),K_1(U)) \oplus \mathrm{Tor}(K_1(A),K_0(U))
  \rightarrow 0, \end{multline*}
  and
  \begin{multline*} 0 \rightarrow K_0(A) \otimes K_1(U) \oplus K_1(A) \otimes K_0(U) \rightarrow K_{1}(A \tensor U) \rightarrow \\ \mathrm{Tor}(K_0(A),K_0(U)) \oplus \mathrm{Tor}(K_1(A),K_1(U))
  \rightarrow 0, \end{multline*}
  now by applying $K_{0}(U)=\mathbb{Q}$ and $K_{1}(U)=0$ into these two sequences, we have
  \begin{align*}
  K_{*}(A) \otimes \mathbb{Q} \cong K_{*}(A \otimes U)
  \end{align*}
  where $*=0,1$.
\end{proof}

\begin{definition}
Consider a collection of complex C*-algebras, and suppose that $U$ is the universal complex UHF-algebra. Define a map $\chi_{2}(A):=m-n$ where $m$ is the $\mathbb{Q}$-dimension of the abelian group $K_1(A \otimes U)$ and $n$ is the $\mathbb{Q}$-dimension of the abelian group $K_0(A \otimes U)$, for every  C*-algebra $A$ in the collection.
\end{definition}

\begin{corollary}\label{additive-chi2}
Consider a collection of complex C*-algebras, then $\chi_{2}$ as  defined above is additive.
\end{corollary}
\begin{proof}
 First consider a short exact sequence
 \begin{align}\label{seq-2}
0 \longrightarrow A \longrightarrow C \longrightarrow B \longrightarrow 0
\end{align}
 and suppose that $U$ is the universal complex UHF-algebra, therefore since complex UHF-algebras are nuclear, tensoring sequence \eqref{seq-2} by $U$ gives rise to again a short exact sequence as follows
  \begin{align*}
0 \longrightarrow A \otimes U \longrightarrow C \otimes U \longrightarrow B\otimes U  \longrightarrow 0
\end{align*}
which induces a six term cyclic exact sequence
\[
\begin{tikzcd}
K_{0}(A \otimes U)  \ar[r, "\varphi_{1}"]& K_{0}(C \otimes U)   \ar[r,"\varphi_{2}"] & K_{0}(B \otimes U) \ar[d,"\varphi_{3}"] \\
K_{1}(B \otimes U) \ar[u,"\varphi_{0}"] & K_{1}(C \otimes U)  \ar[l,"\varphi_{5}"] & K_{1}(A \otimes U)  \ar[l,"\varphi_{4}"]
\end{tikzcd}
\]
but according to the Lemma \eqref{UHF}, $K_{*}(A_{i}\otimes U)$ is a $\mathbb{Q}$-module for every complex C*-algebra $A_{i}$, now applying Lemma \eqref{function} and definition of $\chi_{2}$ implies that $$\chi_{2}(C)=\chi_{2}(B)+\chi_{2}(A).$$
\end{proof}

\section{Characters on Grothendieck rings of complex C*-algebras}
We now construct two characters for Grothendieck rings of complex C*-algebras. In the first case we consider Grothendieck rings of complex C*-algebras with finitely generated K-theory that are in the KTP class, we define $\chi_{1}(A):=m-n$ where $m$ is the number of free generators of the finitely generated abelian group $K_1(A)$ and $n$ is the number of free generators of the finitely generated abelian group $K_0(A)$.

In the second case we consider Grothendieck rings of complex C*-algebras that are in the KTP class and we define $\chi_{2}(A):=m-n$ where $m$ is $\mathbb{Q}$-dimension of abelian group $K_1(A \otimes U)$ and $n$ is $\mathbb{Q}$-dimension of abelian group $K_0(A \otimes U)$, for every  C*-algebra $A$ in the Grothendieck rings and we suppose that $U$ is the universal complex UHF-algebra, notice that $  K_{*}(A \otimes U)\cong  K_{*}(A) \otimes \mathbb{Q} $.

To show that $\chi_{1}$ and $\chi_{2}$ are characters, we must first of all check they are  well-defined on the Grothendieck ring. This means checking compatibility with the equivalence relations. First of all, $K$-theory respects isomorphism of $C^*$-algebras. Secondly, $K$-theory respects (finite) direct sums. Therefore we check the additivity and multiplicativity of the character, meaning if
$$0\longrightarrow A\longrightarrow C \longrightarrow B \longrightarrow 0 $$
is a short exact sequence then $\chi$ must follow the condition $\chi(C)=\chi(A)+\chi(B)$, and we showed in corollary \eqref{additive-chi1} and corollary \eqref{additive-chi2} both $\chi_{1}$ and $\chi_{2}$ are additive. Therefore in order to see that that $ \chi_{1}$ and $\chi_{2}$ are characters we check that they are multiplicative and this means that we need to check $\chi(A \otimes B)=\chi(A)\chi(B)$.

\begin{definition}
Consider the Grothendieck rings of complex C*-algebras then we say that an additive map $\chi$ on the Grothendieck rings is a character if it is multiplicative, in other words $\chi(A\tensor B)=\chi(A)\chi(B)$, where $A$ and $B$ are complex C*-algebras.
\end{definition}

\begin{theorem} Suppose that $A$ and $B$ have finitely generated K-theory and $A$ or $B$ is in the KTP class, then $\chi_{1}(A\tensor B)=\chi_{1}(A)\chi_{1}(B).$
\end{theorem}
	\begin{proof} Under this hypothesis, we have the exact sequences
	 \begin{multline*}0 \rightarrow K_0(A) \otimes K_0(B) \oplus K_1(A) \otimes K_1(B) \rightarrow K_0(A \otimes B) \rightarrow\\ \mathrm{Tor}(K_0(A),K_1(B)) \oplus \mathrm{Tor}(K_1(A),K_0(B)) \rightarrow 0\
	 \end{multline*}
	 and
	 \begin{multline*}0 \rightarrow K_0(A) \otimes K_1(B) \oplus K_1(A) \otimes K_0(B) \rightarrow K_1(A \otimes B) \rightarrow\\ \mathrm{Tor}(K_0(A),K_0(B)) \oplus \mathrm{Tor}(K_1(A),K_1(B)) \rightarrow 0\
	 \end{multline*}
and tensoring with $Q$ we have
$$(K_0(A)\tensor Q) \otimes (K_0(B)\tensor Q) \oplus (K_1(A)\tensor Q) \otimes (K_1(B)\tensor Q) = K_0(A \otimes B)\tensor Q,$$
and
$$(K_0(A)\tensor Q) \otimes (K_1(B)\tensor Q) \oplus (K_1(A)\tensor Q) \otimes (K_0(B)\tensor Q) = K_1(A \otimes B)\tensor Q.$$
But isomorphisms preserve the number of free generators, thus using these two isomorphisms and definition of $\chi_{1}$ contributes into $\chi_{1}(A\tensor B)=\chi_{1}(A)\chi_{1}(B).$
	\end{proof}

\begin{theorem}(Glimm) \label{glimm}
Let $A$ and $B$ be complex UHF-algebras then $A \cong B$ if and only if $(K_{0}(A), [1_{A}]) \cong (K_{0}(B), [1_{B}]).$
\end{theorem}

\begin{lemma} \label{middle of uct}
Let $U$ be the universal UHF-algebra, then $U \otimes U \cong U$.
\end{lemma}
\begin{proof}
By Lemma \eqref{UHF} we have
\begin{align}
  K_{0}(U) \otimes \mathbb{Q} \cong K_{0}(U \otimes U)
  \end{align}
  but since $K_{0}(U)= \mathbb{Q}$ and the tensor product of abelian groups is the ordinary tensor product over $\mathbb{Z}$ we have $ K_{0}(U \otimes U)= \mathbb{Q}$, now the classification of complex UHF-algebras \eqref{glimm} implies that $U \otimes U \cong U$.
\end{proof}

\begin{theorem}
Let $A$ and $B$ be complex C*-algebras and $A$ is in the KTP  class, then $\chi_{2}(A\tensor B)=\chi_{2}(A)\chi_{2}(B).$
\end{theorem}	

\begin{proof}
By Lemma \eqref{middle of uct} since $U \otimes U \cong U$, we have $(A \otimes U)\otimes(B \otimes U)\cong ( A \otimes B)\otimes U$. Since the KTP class is closed under tensor product, by hypothesis $A \otimes U$ is in the $KTP$ class. Therefore KTP theorem \eqref{KTP} holds for $A \otimes U$, which means that we have the following sequences
\begin{multline*}0 \rightarrow K_0(A \otimes U) \otimes K_0(B \otimes U) \oplus K_1(A \otimes U) \otimes K_1(B \otimes U) \rightarrow K_0((A \otimes B)\otimes U) \rightarrow\\ \mathrm{Tor}(K_0(A\otimes U),K_1(B \otimes U)) \oplus \mathrm{Tor}(K_1(A \otimes U),K_0(B \otimes U)) \rightarrow 0
 \end{multline*}
 and
 \begin{multline*}0 \rightarrow K_0(A \otimes U) \otimes K_1(B \otimes U) \oplus K_1(A \otimes U) \otimes K_0(B \otimes U) \rightarrow K_1((A \otimes B)\otimes U) \rightarrow\\ \mathrm{Tor}(K_0(A\otimes U),K_0(B \otimes U)) \oplus \mathrm{Tor}(K_1(A \otimes U),K_1(B \otimes U)) \rightarrow 0
 \end{multline*}
 but $K_*(A \otimes U)$ and $K_*(A \otimes U)$ are torsion free, which arises into two isomorphisms
 $$K_0(A\otimes U) \otimes K_0(B\otimes U) \oplus K_1(A\otimes U)\otimes K_1(B\otimes U) = K_0((A \otimes B)\otimes U),$$
 and
 $$K_0(A\otimes U) \otimes K_1(B\otimes U \oplus (K_1(A\otimes U)\otimes K_0(B\otimes U) = K_1((A \otimes B)\otimes U)$$
now by applying these two isomorphisms and definition of $\chi_{2}$ we get $\chi_{2}(A\tensor B)=\chi_{2}(A)\chi_{2}(B).$
\end{proof}

The following theorem shows that the map $Q$ defined through a F\o lner sequence on $Z^\infty$ is in fact a character, and this is one of our main results.
\begin{theorem} \label{th:countably.generated.case}
Suppose that $A$ and $B$ have countably generated K-theory and $A$ or $B$ is in the KTP class, then $Q(A\tensor B)=Q(A)Q(B).$
\end{theorem}
	\begin{proof}
Recall from \cref{linear map}  that\begin{align*}
 Q(A\tensor B) = q( k_{0}(A\tensor B))-q(K_{1}(A\tensor B)).
\end{align*}
The K\"unneth tensor product theorem gives the exact sequences
	 \begin{multline*}0 \rightarrow K_0(A) \otimes K_0(B) \oplus K_1(A) \otimes K_1(B) \rightarrow K_0(A \otimes B) \rightarrow\\ \mathrm{Tor}(K_0(A),K_1(B)) \oplus \mathrm{Tor}(K_1(A),K_0(B)) \rightarrow 0,\
	 \end{multline*}
	 and
	 \begin{multline*}0 \rightarrow K_0(A) \otimes K_1(B) \oplus K_1(A) \otimes K_0(B) \rightarrow K_1(A \otimes B) \rightarrow\\ \mathrm{Tor}(K_0(A),K_0(B)) \oplus \mathrm{Tor}(K_1(A),K_1(B)) \rightarrow 0.\
	 \end{multline*}
In both of these sequences, the third term is an at most countably generated abelian group with only torsion elements, and, by \cref{lemma:TorsionDoesntMatter}, such a group makes a negligible contribution in the limit.
 Indeed, this means exactly that for both of these sequences, the set of generators of the free part has the same cardinality in the first and the second terms. The function $Q$ takes a limit of the difference of the number of free generators of the $K$-theory, and thus
  $ Q(A\tensor B)$ equals
 \begin{align*}
 &\lim_{n \rightarrow \infty} \frac{\ln(|(K_0(A) \otimes K_0(B) \oplus K_1(A) \otimes K_1(B))\cap \mathcal G_n|) -
                                                                                    \ln(|(K_0(A) \otimes K_1(B) \oplus K_1(A) \otimes K_0(B))\cap \mathcal G_n|) }{\ln(|\mathcal G_n|) }\\
 &=\lim_{n \rightarrow \infty} \frac{\ln(|(K_0(A)\cap \mathcal G_n|)-\ln(|(K_1(A)\cap \mathcal G_n|) }{\ln(|\mathcal G_n|)} \lim_{n \rightarrow \infty} \frac{\ln(|(K_0(B)\cap \mathcal G_n|)-\ln(|(K_1(B)\cap \mathcal G_n|) }{\ln(|\mathcal G_n|)}\\
 &=(q(K_{0}(A)-q(K_{1}(A)) (q(K_{0}(B)-q(K_{1}(B))\\
 &=Q(A)Q(B).
\end{align*}
	\end{proof}
\begin{theorem}
For a C*-algebra $A$ with K-groups $K_{0}(A)=\mathbb{Q}$ and $K_{1}(A)=0$, the character $Q (A)$ of $A$ is zero.
\end{theorem}
\begin{proof}
According to Lemma \eqref{lemma: q is additive}, we know that $q$ is additive on a short exact sequence. By applying Lemma \eqref{function} and using the six-term cyclic exact sequence, we obtain $Q(A) = q(K_{0}(A)) - q(K_{1}(A))$. Since we assume that $K_{0}(A) = \mathbb{Q}$ and $K_{1}(A) = 0$, we have $Q(A) = q(\mathbb{Q})$.
Consider the short exact sequence
\[ 0 \longrightarrow \mathbb{Z} \rightarrow \mathbb{Q} \rightarrow \mathbb{Q}/\mathbb{Z} \longrightarrow 0. \]
Using the additivity of $q$, we have $q(\mathbb{Q}) = q(\mathbb{Z}) + q(\mathbb{Q}/\mathbb{Z})$. Since $\mathbb{Q}/\mathbb{Z}$ is a countably generated torsion group , we can deduce that, by Lemma \eqref{lemma:TorsionDoesntMatter}, $q(\mathbb{Q}/\mathbb{Z}) = 0$. Furthermore, it is clear that $q(\mathbb{Z}) = 0$.
Therefore, we have $Q(A) = q(\mathbb{Q}) = q(\mathbb{Z}) + q(\mathbb{Q}/\mathbb{Z}) = 0 + 0 = 0$, which completes the proof.
\end{proof}
We have thus exhibited a selection of additive and multiplicative maps that factor through K-theory. It is an open question whether there exist such maps not factoring through K-theory.

Naturally, the kernel of the above character provides us with an ideal of the Grothendieck ring that does not obviously come from an idempotent.

%
%
	

\section{Idempotents}
	
	For each idempotent $j=j\tensor j,$ there exists a corresponding ring homomorphism of a Grothendieck ring into a ring ideal, given by $x\mapsto x \tensor j.$ (Proposition \ref{prop:ideal} shows that however there are interesting ideals that are not of this form.)  Because such maps can give a lot of algebraic information on the ring,  it appears interesting to have some information on known idempotents. The two best known such idempotents are probably $K$, the compact operators, and $Z,$ the Jiang-Su algebra. From a historical point of view, $K$, the compact operators, is one of the very first operator algebras to be studied. Eventually, in the classification program, it was observed that in a given classifiable class of C*-algebras, such as, say, the AF algebras, the projectionless simple algebras, or the purely infinite simple algebras, it seems interesting to single out the algebras having the same invariants as the complex numbers, $\C.$ In most cases, such algebras turned out to be self-absorbing, in the sense that, for example, $Z\tensor Z$ is isomorphic to $Z.$ In fact, these algebras have a stronger property of being strongly self-absorbing. At this point, we digress for a moment to give a very short introduction to the classification program. Classification theory tends to focus on simple C*-algebras. The main idea is that inside a suitably defined class of C*-algebras, some version of K-theory may provide a complete invariant. Finding such a class can be difficult, as demonstrated by the TAF algebras. From our viewpoint, however, we are interested mostly in the tensor idempotents found as a byproduct of this program. Our viewpoint is not, however, a complete departure from the classification theory point of view, because from the viewpoint of that theory, it occurs several times that the C*-algebras belonging to a ring ideal generated by a properly chosen  idempotent have good behaviour from the classification point of view \cite{PrimesInClassification}.
	The list of special idempotents that have this kind of connection with the classification program is: $Z$, the Jiang-Su algebra;
 $O_2$ and $O_\infty,$ the Cuntz algebras; UHF algebras with all primes of infinite multiplicity; and of course $\C$ the algebra of complex numbers. Furthermore, it is conjectured that the projectionless Jacelon--Razak algebra, $R,$ is an idempotent. On the other hand, the simple and very classical idempotent $K,$ the compact operators, though it is very important as a basic example of an infinite-dimensional C*-algebra and an idempotent, does not really belong to the above list. Additionally, there is a very large class of idempotents defined by infinite tensor products.  In fact, the UHF algebras are already of this general form, so one can see that these classes mentioned are not disjoint. It is shown in \cite{blackadar1977} that infinite tensor products of separable algebras are separable, and it can be seen that they are nuclear if the tensor factors are nuclear. Thus, the infinite tensor product $\bigotimes_1^\infty A$ is clearly an idempotent. We will denote elements of this large class of idempotents by $A_\infty.$ Under mild technical conditions \cite{InfiniteTensorProductsAreZstable}, they absorb $Z,$ so that usually we have $Z\tensor A_\infty = A_\infty.$ An idempotent need not be simple as a C*-algebra. An especially interesting example of an idempotent that is not simple as a C*-algebra is $C([0,1])_\infty,$ in other words, $C(X)$ where $X$ is the Tychonoff cube of weight $\aleph_0.$
 We remark that in the category of C*-algebras the answer to Ulam's problem is negative: in other words, it is not true that $A\tensor A \isom B\tensor B$ implies that $A$ is isomorphic to $B.$ There is a somewhat involved counterexample in the unital and commutative case already \cite{NegativeUlam}.
	
	There exist several classes, loosely speaking, of idempotents. One can mention the unital, strongly self-absorbing idempotents. It has been shown  \cite{StronglySelfAbsAreZstable,StronglySelfAbsAreZstable0} that all so-called strongly self-absorbing simple idempotents have the property of absorbing $Z,$ meaning that $W\tensor Z \isom W$ for such unital idempotents $W.$ Kirchberg has shown that $O_2$ absorbs all the known unital simple separable and nuclear idempotents \cite{KirchToAppear}, meaning that for such idempotents $W\tensor O_2\isom O_2.$ 	In some cases, though, tensoring two idempotents simply gives a new idempotent, as is the case with nearly every tensor product with $K,$ for example. The complex numbers, $\C,$ regarded as a C*-algebra, certainly are an idempotent, and also provide a multiplicative unit for a Grothendieck ring that they are in.
	The following Hasse diagram summarizes the known relationships between simple idempotents:
	
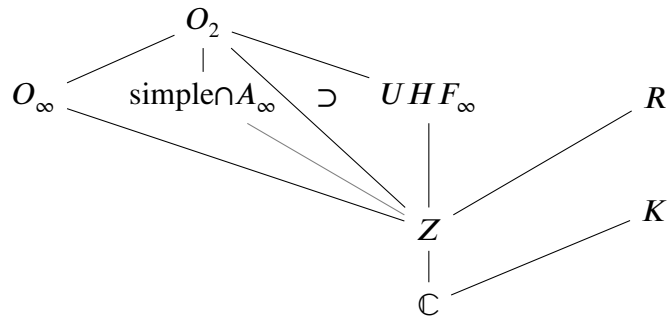
\begin{center}
\begin{tikzpicture}[yscale=0.5,xscale=1.5]
  \node (Oinf) at (-3.5,2) {$O_\infty$};
  \node (Otwo) at (-2,4) {$O_2$};
  \node (Ainf) at (-2,2) {simple$\cap A_\infty$}; \node (inclusion) at (-0.9,2) {$\supset$};
  \node (Uinf) at (0,2) {$UHF_\infty$};
  \node (J) at (2,2) {$R$};
  \node (Z) at (0,-1.5) {$Z$};
	\node (C) at (0,-3.5) {$\C$};
	\node (K) at (2,-1) {$K$};
  \draw (Oinf) -- (Otwo); \draw (C) -- (K); \draw (Z) -- (C); \draw (Z) to (Oinf);
	\draw (Z) -- (Otwo); \draw [gray] (Z) to (Ainf); \draw (Z) to (Uinf); \draw (Z) to (J);
	\draw (Otwo) to (Ainf); \draw (Otwo) to (Uinf);
\end{tikzpicture}\end{center}

The trivial algebra  $\{0\}$ is simple and idempotent but is omitted from the diagram. This concludes our discussion of special idempotents. If we turn to the general case of the collection of all idempotents of a Grothendieck ring, we note that the idempotents of a commutative ring with unit form a Boolean lattice. (The unit here is given by the class of $\C$, the algebra of complex numbers.) By
Stone's famous theorem these idempotents are thus isomorphic to the collection of closed and open sets of some topological space, a so-called Stonean space \cite{OnStones}. Stonean spaces are very well behaved from the topological point of view, so in a sense this provides a structure theorem for the idempotents of a Grothendieck ring.
\section{Computing the Grothendieck ring of C*-algebras}

We now  compute the Grothendieck ring of some classes of C*-algebras. This is only sometimes feasible. The main cases we look at are those with finitely many ideals, and those with finitely generated K-theory.
We begin with reminding the reader that the most general case, namely the ring generated by all C*-algebras, is trivial. This was shown at the start of Section 2. If we wish to consider other classes, the first issue is closure under the operations needed in the Grothendieck construction, namely isomorphism, finite direct sum, and formation of extensions. Secondly, the KTP property seemed important for being able to construct characters.
One would hope that the KTP property would pass to ideals and to quotients, as did the property of being an AF algebra in our first examples, but this seems unclear. Precisely because of such issues, one prefers to work with the bootstrap class, which has the property, by definition, of having a 2-of-3 property under the formation of extensions, and this is sufficient for a reasonable theory. Thus, if a C*-algebra has the bootstrap property, then forming quotients by ideals with the bootstrap property will ensure that the quotient has the bootstrap property.

\subsection{The Grothendieck ring of separable nuclear C*-algebras with finitely generated K-theory.}
The triviality argument now goes through just in case that the algebra $A$ has $K_* (A)=\{e\},$ for in this case the infinite direct sum  $D:=\bigoplus^\infty_1 A_i $ has trivial, thus finitely generated, $K$-theory. This is because $K$-theory respects both finite direct sums and countable inductive limits, and thus respects infinite direct sums. Separability of the algebra $D$ follows in essentially the same way: a countable  direct sum is a limit of finite direct sums, and each of the finite direct sums is separable (or more directly, we can use an approximation result together with Cantor's result that a countable union of countable sets is countable). Nuclearity is an approximation property that is preserved under infinite direct sums. Thus $A\oplus D$ is isomorphic to $D$, and this implies that algebras with vanishing K-theory are zero in this Grothendieck ring.
				
In this ring, we do not know if the character $\chi$ is multiplicative, because our previous proof of this depended on the KTP. The KTP conjecture would however imply that this character is multiplicative.

	\subsection{The Grothendieck ring of separable, nuclear, C*-algebras with finitely many ideals.}
	
	In this ring, the fundamental triviality argument is indeed trivial, because infinite direct sums or products have infinitely many ideals, given by the sequences which are zero in the $n^{th}$ place. The following lemma shows that the class is closed under the minimal tensor product.
\begin{lemma}
 The class of C*-algebras with finitely many ideals is closed under the tensor product $\otimes_{min}.$
\end{lemma} \begin{proof} Quite generally, there exists a map from $Prim A \times Prim B$ to $Prim A\otimes_{min} B$ that has dense range \cite[II.9.6.7]{blackadarbook2}. But if a finite set is dense, in $Prim A\otimes_{min} B,$ then it must be that there are finitely many primitive ideals, and hence finitely many ideals, in $A\otimes_{min} B.$
 \end{proof}
Proposition \ref{prop:ideal} shows that UCT class algebras with vanishing $K$-theory form an ideal, $I.$ The proof of the UCT shows that this ideal can be alternatively characterized as the ideal generated by the separable nuclear $KK$-contractible algebras.
	
\subsection{Computing the Grothendieck ring of abelian C*-algebras}
The Gelfand transform establishes an equivalence between abelian C*-algebras and completely regular locally compact topological spaces. As a result, the Grothendieck ring generated by abelian C*-algebras is equivalent to a ring generated by these topological spaces, using disjoint union and Cartesian product as operations. The equivalence relation in the context of the Grothendieck ring incorporates two key properties. Firstly, it treats homeomorphic spaces as equal. Secondly, if $A$ represents an ideal in $B$ with a quotient space $C$, then the relation $[B]=[A]\oplus [C]$ holds. This implies that a ideal can be understood as a direct summand within the ring structure.

\section{real C*-algebras}
 A real C*-algebra $A$ is a real Banach *-algebra which satisfies the C*-equation $\|  a^{*}a\|=\| a\|^{2}$ as well as the axiom that $1+a^{*}a$ is invertible in the unitization $A^{+}$ for every $a \in A$. This is equivalent to saying that $A$ is *-isometrically isomorphic to a norm-closed adjoint-closed algebra of operators on a real Hilbert space \cite{Palmer}. Note that every complex C*-algebra can be considered as a real C*-algebra by forgetting the complex structure.

If $A$ is any real C*-algebra, it has a unique complexification $A_{\mathbb{C}}= A \otimes _{\mathbb{R}} \mathbb{C}$ which is a complex C*-algebra. Once we have the complexification of a real C*-algebra $A$, we can recover $A$ using an involutive *- anti-automorphism on $A_{\mathbb{C}}$. For example, consider the involutive *- anti-automorphism
$$
\alpha : A_{\mathbb{C}} \ni x+iy \mapsto x^{*} +i y^{*} \in A_{\mathbb{C}}
$$
then $A = \{ a \in A_{\mathbb{C}} \ | \  \alpha(a)=a^{*} \}$ is a real C*-algebra.

 \begin{definition}
 A complex C*-algebra $A$ has a real structure if it carries an involutive *- anti-automorphism.
 \end{definition}
Therefore, by the definition of real structures if we are given a complex C*-algebra $A$ and an involutive *- anti-automorphism on it, then we can recover a real C*-algebra the same as above which is called the real structure and its complexification is the complex C*-algebra $A$. However, if $A$ is a complex C*-algebra, in general, it does not admit a *- anti-automorphism . In other words, if $A$ is a complex C*-algebra then it is not the complexification of a real C*-algebra in general. For example, Connes \cite{Connes} showed the existence of factors acting on a seperable Hilbert space which is not anti isomorphic to itself, therefore it cannot admit any real structure.

Now let $A$ be a complex C*-algebra that admits a real structure, then in general there might be more than one real structure. For example, consider $M_{2}(\mathbb{C})$ as a complex C*-algebra. There are two real structures on  $M_{2}(\mathbb{C})$ which are not isomorphic. One of the real structures can be recovered by using the involutive anti-automorphism $\alpha_{1}$ which is the transpose operation on $M_{2}(\mathbb{C})$, then we have $\{a \in   M_{2}(\mathbb{C}) : \alpha_{1}(a)=a^{*}\}=M_{2}(\mathbb{R})$. Now consider the involutive anti-automorphism $\alpha_{2}$ defined by
 \[
   \alpha_{2} \left[ {\begin{array}{cc}
   a & b \\
   c & d \\
  \end{array} } \right]=
  \left[ {\begin{array}{cc}
   d & -b \\
   -c & a \\
  \end{array} } \right]
\]
then the real structure $\{a \in   M_{2}(\mathbb{C}) : \alpha_{2}(a)=a^{*}\}$ is $\mathbb{H}$. Now for showing that $M_{2}(\mathbb{R})$ and $\mathbb{H}$ are not isomorphic as real C*-algebras we use Giordano's classification of real AF-algebras \cite{Gior}, indeed, he has obtained an invariant which is equivalent to one by Goodearl and Handelman \cite{handel}. Giordano's invariant contains $KO_{0}$, $KO_{2}$ and $KO_{4}$, therefore since $KO_{2}(M_{2}(\mathbb{R})) \neq KO_{2}(\mathbb{H})$ we conclude that $M_{2}(\mathbb{R})$ and $\mathbb{H}$ are two distinct real structures of the complex C*-algebra $M_{2}(\mathbb{C})$.

On the other hand, there are some cases for which there exists a unique real structure. For example, the hyperfinite $II_{1}$ factor $R$ has a unique real structure (\cite{Erling}, \cite{G}, \cite{Gior}), and also it is shown that the injective factor $II_{\infty}$ has a unique real structure \cite{Gior}.
Now consider separable nuclear purely infinite simple C*-algebras which satisfy the Universal Coefficient Theorem in both real and complex cases. There are classification theorems for these kind of C*-algebras.
\begin{theorem} \cite{philips} \label{comp_classification}
Let $A$ and $B$ be separable nuclear purely infinite simple complex C*-algebras which are in the UCT class. If $K_{*}(A)\cong K_{*}(B)$, then $A \cong B$.
\end{theorem}
Moreover, there is a classification for the real case by Boersema, Ruiz and Stacy as follows
\begin{theorem} \cite{Stacy}\label{real_classification}
Let $A$ and $B$ be separable nuclear purely infinite simple real C*-algebras which satisfy the Universal Coefficient Theorem such that $A_{\mathbb{C}}$ and $B_{\mathbb{C}}$ are also simple. Then $A$ and $B$ are isomorphic if and only if $K^{\text{CRT}}(A) \cong K^{\text{CRT}}(B)$.
\end{theorem}

Now let us define the united K-theory which consists of three functors real K-theory, complex K-theory and self-conjugate K-theory. United K-theory is a device that Boersema \cite{Boersema} used to construct a K\"unneth type sequence of real C*-algebras. Indeed, united K-theory solves the problem of projective dimension one in the category of $K_{*}(\mathbb{R})$-modules.
\begin{definition}\label{united K-th}
	Let $A$ be a real C*-algebra, then the united K-theory of $A$,  $K^{\text{CRT}}(A)$ is a triple of $\mathbb{Z}$-graded abelian groups
	\begin{align*}
	K^{\text{CRT}}(A):= \{ KO_{*}(A), KU_{*}(A), KT_{*}(A) \},
	\end{align*}
where $KO_{*}(A)$ is the ordinary real K-theory which has period $8$, $KU_{*}(A)= K_{*}(\mathbb{C} \otimes A)$ is the complex K-theory with period $2$, and $KT_{*}(A)=K_{*}(T \otimes A)$ is the self-conjugate K-theory and it has period $4$, where $T$ is the real C*-algebra defined as
\begin{align*}
T=\{f \in C([0,1],\mathbb{C}) | f(0)=\overline{f(1)}  \}
\end{align*}
then if $A$ is any other C*-algebra
\begin{align*}
T \otimes A\cong \{f \in C([0,1],\mathbb{C} \otimes A) | f(0)=\overline{f(1)}  \}
\end{align*}
the conjugation of $\mathbb{C} \otimes A$ is defined by $ \overline{\lambda \otimes a} = \overline{\lambda} \otimes a$.
\end{definition}

An object $M=\{M^{O}, M^{U}, M^{T}  \}$ in the category $\text{CRT}$ is said to be acyclic if the sequences in \cite{Boersema} are exact. Also in \cite{Boersema} it is shown that the united K-theory defined in \eqref{united K-th} is acyclic. But acyclic objects can be used in the united K-theory concept to show that if a map is isomorphism the other two are also isomorphisms.
\begin{lemma} \cite{Boersema} \label{acyclic}
Suppose that $\phi: M \rightarrow N$ is a map of acyclic objects of $\text{CRT}$. If one of $\phi^{O}$, $\phi^{U}$ and $\phi^{T}$ is an isomorphism, then the other two are also isomorphisms.
\end{lemma}
Now we use the classification Theorems \eqref{comp_classification} and \eqref{real_classification} to prove that if there is a real structure for a separable nuclear purely infinite simple complex C*-algebras which are in the UCT class, then it is unique up to isomorphism.
\begin{theorem}
Let $A$ be a complex separable nuclear purely infinite simple C* algebra which is in the UCT class, and  $B$ be a real separable nuclear purely infinite simple C* algebra in the UCT class. Suppose that $A$ and $B$ have the same complex K-theory then $B$ is a unique real structure of $A$ up to isomorphism.
\end{theorem}
\begin{proof}
By hypothesis $K_{*}(A)=KU_{*}(B)$ which means that $K_{*}(A)=K_{*}(B \otimes \mathbb{C})$, now by classification Theorem \eqref{comp_classification} we have $A \cong B \otimes \mathbb{C} $ which means that $A$ is the complexification of the real C*-algebra $B$. For showing the uniqueness, suppose that there are $B_{1}$ and $B_{2}$ real C*-algebras under the hypothesis such that $K_{*}(A)=KU_{*}(B_{1})$ and $K_{*}(A)=KU_{*}(B_{2})$ therefore $KU_{*}(B_{1})=KU_{*}(B_{2})$, but, $K^{\text{CRT}}$ is acyclic so by Lemma \eqref{acyclic} $KO_{*}(B_{1})=KO_{*}(B_{2})$ and $KT_{*}(B_{1})=KT_{*}(B_{2})$ which implies $K^{\text{CRT}}(B_{1})=K^{\text{CRT}}(B_{2})$ and then by classification Theorem \eqref{real_classification} we have $B_{1} \cong B_{2}$.
\end{proof}

 Following this theorem, we will show in Corollary \eqref{real structe of UHF} that every complex UHF-algebra has a unique real UHF structure.

\section{Characters of Grothendieck rings of real C*-algebras}
For the Grothendieck rings of real C*-algebras we introduce two candidate  characters. Let us define two different additive maps on a collection of real C*-algebras. For the multiplicativity we need to use the K\"unneth theorem, but the K\"unneth theorem for complex C*-algebras cannot be adapted to real C*-algebras.
For real C*-algebras a universal coefficient theorem was proven in \cite{Boersema} using united K-theory. We take advantage of these developments to provide characters for the the Grothendieck rings of real C*-algebras.

Recall that $N$ is the bootstrap category of complex C*-algebras defined in Section $3$.

\begin{theorem} [\protect{K\"unneth theorem \cite{Boersema}}]
Let $A$ and $B$ be real C*-algebras such that $ A \otimes \mathbb{C} \in N$. Then there is a natural short exact sequence
 \begin{align*}
 0 \rightarrow K^{\text{CRT}}_{*}(A) \otimes_{\text{CRT}} K^{\text{CRT}}_{*}(B) \xrightarrow{\alpha}  K^{\text{CRT}}_{*}(A \otimes B) \xrightarrow{\beta} \mathrm{Tor}_{\text{CRT}} ( K^{\text{CRT}}_{*}(A),K^{\text{CRT}}_{*}(B)) \rightarrow 0
 \end{align*}
 which does not necessarily split.
\end{theorem}
Let $N_{\mathbb{R}}$ be the bootstrap category of real C*-algebras, which is the smallest subcategory of real C*-algebras whose complexification is nuclear, contains real separable type $I$ C*-algebras, is closed under the operation of taking inductive limits, stable isomorphisms, and crossed product by $\mathbb{Z}$ and $\mathbb{R}$, and satisfies the two out of three rule for short exact sequences, then we have the following theorem:

\begin{theorem}

Let $A$ and $B$ be real C*-algebras and $A \in N_{\mathbb{R}}$. Then there is a natural short exact sequence
 \begin{align*}
 0 \rightarrow K^{\text{CRT}}_{*}(A) \otimes_{\text{CRT}} K^{\text{CRT}}_{*}(B) \xrightarrow{\alpha}  K^{\text{CRT}}_{*}(A \otimes B) \xrightarrow{\beta} \mathrm{Tor}_{\text{CRT}}( K^{\text{CRT}}_{*}(A),K^{\text{CRT}}_{*}(B)) \rightarrow 0.
 \end{align*}
 \end{theorem}
 The map $\alpha$ is a $\text{CRT}$-module homomorphism and is defined by
 \begin{align*}
 \alpha_{O}: KO_{m}(A) \otimes KO_{n}(B) \rightarrow KO_{m+n}(A \otimes B)
 \end{align*}
  \begin{align*}
 \alpha_{U}: KU_{m}(A) \otimes KU_{n}(B) \rightarrow KU_{m+n}(A \otimes B)
 \end{align*}
  \begin{align*}
 \alpha_{T}: KT_{m}(A) \otimes KT_{n}(B) \rightarrow KT_{m+n}(A \otimes B)
 \end{align*}
 and the map $\beta$ is a $\text{CRT}$-homomorphism of degree -1.
Now consider $K^{\text{CRT}}$, $KO$, $KU$, and $KT$ we define the even and odd parts of these K-theories as follows
 $$K^{\text{CRT}}_{i}:=KO_{(i  \mod 8)} \oplus KU_{(i \mod 2)} \oplus  KT_{(i \mod 4)} $$
 for every $i$ in $\mathbb{Z}$,  and
 \begin{align*}
 &K^{\text{CRT}}_{even}:=K^{\text{CRT}}_{0} \oplus K^{\text{CRT}}_{2} \oplus K^{\text{CRT}}_{4} \oplus K^{\text{CRT}}_{6}, \\
&K^{\text{CRT}}_{odd}:=K^{\text{CRT}}_{1} \oplus K^{\text{CRT}}_{3} \oplus K^{\text{CRT}}_{5} \oplus K^{\text{CRT}}_{7},\\
&KO_{even}:=KO_{0} \oplus KO_{2} \oplus KO_{4} \oplus KO_{6}, \\ &KO_{odd}:=KO_{1} \oplus KO_{3} \oplus KO_{5} \oplus KO_{7}, \\
& KT_{even}:=KT_{0} \oplus KT_{2}, \\
&KT_{odd:}=KT_{1} \oplus KT_{3}, \\
&KU_{even}:=KU_{0}, \\
&KU_{odd}:=KU_{1}.
 \end{align*}

\begin{definition} \label{real character}
Consider a collection of real C*-algebras, then for every $A$ in this collection for which $KO(A)$, $KU(A)$ and $KT(A)$ are finitely generated define the map $\chi_{1}^{R}(A):=m-n$ where $m$ and $n$ are defined as follows

$m$ $=$ the number of free generators of $\ KO_{odd}(A)$ $-$ the number of free generators of $KU_{odd}(A)$ $+$ the number of free generators of $KT_{odd}(A)$

and

$n$ $=$ the number of free generators of $KO_{even}(A)$ $-$ the number of free generators of $KU_{even}(A)$ $+$ the number of free generators of $KT_{even}(A)$.
\end{definition}

\begin{theorem} \label{24-length}
For any short exact sequence
\begin{align*}
0 \rightarrow A \rightarrow C \rightarrow B \rightarrow 0
\end{align*}
of real (complex) C*-algebras, there is a long exact sequence in K-theory

\begin{align*}
... \rightarrow K_{n}(A) \rightarrow K_{n}(C) \rightarrow K_{n}(B) \rightarrow K_{n-1}(A) \rightarrow K_{n-1}(C) \rightarrow ... .
\end{align*}
\end{theorem}
\begin{proof}
Given the identical nature of the proofs in both the real and complex cases, we refer to \cite{blackadar}.
\end{proof}

By incorporating the Bott periodicity theorem, which exhibits periodicity of $8$ for $KO$, $2$ for $KU$, and $4$ for $KT$, we are able to establish a cyclic sequence of length $24$ within $K_*^{\text{CRT}}$.

\begin{theorem} \label{cyclic}
Let $K^{\text{CRT}}_{even}$ and $K^{\text{CRT}}_{odd}$ are defined as above. Then there is a  cyclic exact sequence
\[
\begin{tikzcd}
K^{\text{CRT}}_{even}(A) \ar{r} & K^{\text{CRT}}_{even}(C) \ar{r} & K^{\text{CRT}}_{even}(B)\ar{d} \\
K^{\text{CRT}}_{odd}(B) \ar{u} & K^{\text{CRT}}_{odd}(C) \ar{l} & K^{\text{CRT}}_{odd}(A) \ar{l}
\end{tikzcd}
\]
\end{theorem}
\begin{proof}
Apply Theorem \eqref{24-length}, and Lemma \eqref{Folding}.
 \end{proof}

Now define $\overline{S}(K^{\text{CRT}}_{even}(A)):= S(K^{\text{CRT}}_{0}(A)) \oplus S(K^{\text{CRT}}_{2}(A)) \oplus S(K^{\text{CRT}}_{4}(A)) \oplus S(K^{\text{CRT}}_{6}(A))  $ and $\overline{S}(K^{\text{CRT}}_{odd}(A)):= S(K^{\text{CRT}}_{1}(A)) \oplus S(K^{\text{CRT}}_{3}(A)) \oplus S(K^{\text{CRT}}_{5}(A)) \oplus S(K^{\text{CRT}}_{7}(A))  $ where $S$ is a function defined in Definition \eqref{functionS} .

\begin{lemma} \label{S}
Let $S$ and $\overline{S}$ defined as above, then we have
\begin{align*}
(\overline{S}(K^{\text{CRT}}_{odd}(C)-\overline{S}(K^{\text{CRT}}_{even}(C))=(\overline{S}(K^{\text{CRT}}_{odd}(A)-\overline{S}(K^{\text{CRT}}_{even}(A))+(\overline{S}(K^{\text{CRT}}_{odd}(B)-\overline{S}(K^{\text{CRT}}_{even}(B)).
\end{align*}
\end{lemma}

\begin{theorem}
Consider a collection of real C*-algebras with $KO(A)$, $KU(A)$ and $KT(A)$ finitely generated. Then $\chi_{1}^{R}$ we defined above is additive, in the sense that if
\begin{align*}
0 \rightarrow A \rightarrow C \rightarrow B \rightarrow 0
\end{align*}
is a short exact sequence then $\chi_{1}^{R}(C)=\chi_{1}^{R}(A)+\chi_{1}^{R}(B)$.
\end{theorem}

\begin{proof}
Suppose that
\begin{align*}
0 \rightarrow A \rightarrow C \rightarrow B \rightarrow 0
\end{align*}
is a short exact sequence of real C*-algebras with $KO(A)$, $KU(A)$ and $KT(A)$ finitely generated. According to Theorem \eqref{cyclic} there exists a cyclic exact sequence
\[
\begin{tikzcd}
K^{\text{CRT}}_{even}(A) \ar{r} & K^{\text{CRT}}_{even}(C) \ar{r} & K^{\text{CRT}}_{even}(B)\ar{d} \\
K^{\text{CRT}}_{odd}(B) \ar{u} & K^{\text{CRT}}_{odd}(C) \ar{l} & K^{\text{CRT}}_{odd}(A) \ar{l}
\end{tikzcd}
\]

In order to recover the free parts, we tensor the above cyclic sequence by $\mathbb{Q}$ and then apply Lemma \eqref{function} and then Lemma \eqref{S}.
\end{proof}
In our second definition of an additive map on a collection of real C*-algebras we use the universal real UHF-algebra. Real UHF-algebras are the inductive limit of the sequence of real matrix algebras. The universal real UHF-algebra $U_{r}$ is a real UHF-algebra and will play a role analogous to the universal complex UHF-algebra $U$. The united K-theory of $ U_{r}$ is given as follows \cite{BoersemaUHF}

\begin{table}[!h]
\begin{center}
\begin{tabular}{l l l l l l l l l }
   \label{p1}&0   & 1&2&3 &4 &5 &6&7 \\
    $KO_{*}(U_{r})$ &$\mathbb{Q}$ & 0& 0& 0 & $\mathbb{Q}$ &0 & 0& 0\\
    $KU_{*}(U_{r})$ &$\mathbb{Q}$ & 0& $\mathbb{Q}$& 0 & $\mathbb{Q}$ &0 & $\mathbb{Q}$& 0\\
    $KT_{*}(U_{r})$ &$\mathbb{Q}$ & 0& 0& $\mathbb{Q}$& $\mathbb{Q}$ &0 & 0&$\mathbb{Q}$\\
\end{tabular}
\end{center}
\end{table}

We recall that every real UHF-algebra is in the UCT class and the proof is the same as in the complex case.

\begin{lemma} \label{R.free ab}
Let $U_{r}$ be the universal real UHF-algebra and $A$ be a real C*-algebra then $K_{*}^{\text{CRT}}( A \otimes U_{r} )$ is a $\mathbb{Q}$-module.
\end{lemma}

\begin{proof}
Since $U_{r}$ is in the UCT class then there is a short exact sequence
\begin{align*}
 0 \rightarrow K^{\text{CRT}}_{*}(A) \otimes_{\text{CRT}} K^{\text{CRT}}_{*}(U_{r}) \xrightarrow{\alpha}  K^{\text{CRT}}_{*}(A \otimes U_{r}) \xrightarrow{\beta} \mathrm{Tor}_{\text{CRT}}( K^{\text{CRT}}_{*}(A),K^{\text{CRT}}_{*}(U_{r})) \rightarrow 0
 \end{align*}
 but according to the above table $K^{\text{CRT}}_{*}(U_{r})$ is either $\mathbb{Q}$ or $0$, therefore the map $\alpha$ is an isomorphism which means that
 \begin{align*}
 K^{\text{CRT}}_{*}(A) \otimes_{\text{CRT}} K^{\text{CRT}}_{*}(U_{r}) \cong K^{\text{CRT}}_{*}(A \otimes U_{r})
 \end{align*}
 from this we can conclude that $K^{\text{CRT}}_{*}(A \otimes U_{r})$ is isomorphic to a $\mathbb{Q}$-module.
 \end{proof}
 Real AF-algebras were classified by Giordano  using an invariant consisting of $K_{0}(A)$, $K_{2}(A)$, $K_{4}(A)$  and an order structure on $K_{0}(A) \oplus K_{2}(A)$, where $A$ is an AF-algebra.
 \begin{theorem} \cite{Gior} \label{giordano}
 Let $A$ and $B$ be real AF algebras. if for $i=0,2$ and 4, $\psi_{i}: K_{i}(A) \rightarrow K_{i}(B)$ are group isomorphisms such that $\{ \psi_{0}, \psi_{2}, \psi_{4} \} \in \Psi(A,B)$ and $\{ \Psi_{0}^{-1}, \Psi_{2}^{-1}, \Psi_{4}^{-1} \} \in \Psi(B,A)$, then there is an isomorphism $\alpha:A \rightarrow B$ such that $K_{*}(\alpha)=\Psi_{*}$.
 \end{theorem}

 \begin{lemma} \cite{Schroder} \label{Schroder}
Let $A$ be a real UHF-algebra. Then we have
\begin{align}
    K_{n}(A) =
    \begin{cases}
        \mathbb{Z}\left[\frac{1}{s}\right], & \text{if } n = 0, 4, \\
        \mathbb{Z}_{2}, & \text{if } n = 1, 2 \text{ and } \mathbb{Z}\left[\frac{1}{2}\right] \not\subset K_{0}(A), \\
        0, & \text{otherwise.}
    \end{cases}
\end{align}
where $\mathbb{Z}\left[\frac{1}{s}\right]$ is an additive subgroup of $\mathbb{Q}$ associated to the supernatural number $s$, which is a formal infinite prime factorization, i.e., $s = \prod_{j=1}^{\infty} p^{s_{j}}$ with $s_{j} \in \mathbb{Z}_{+} \cup \{\infty\}$, defined by
\begin{align*}
    \mathbb{Z}\left[\frac{1}{s}\right] = \left\{ \frac{r}{q} \in \mathbb{Q} \mid r \in \mathbb{Z}, q \in \mathbb{N}, q | s \right\}.
\end{align*}
\end{lemma}
In the following lemma we show that for $n=1,2$ if $\mathbb{Z}[\frac{1}{2}] \not\subset K_{0}(A)$ or equivalently if there are infinitely many even maps or in other words infinitely many zero maps then $K_{1}(A)=K_{2}(A)=0$.

\begin{lemma} \label{even map}
Let $A$ be a real UHF-algebra and suppose that there are infinitely many even maps then $K_{1}(A)=K_{2}(A)=0$.
\end{lemma}

\begin{proof}
Since $K_{*}(\mathbb{R})=\mathbb{Z}_{2}$ for $*=1$ and $2$, then $K_{1}(A)$ and $K_{2}(A)$ is the inductive limit of the sequence:

\begin{align}  \label{seq}
\mathbb{Z}_{2} \xrightarrow{\times d_{1}} \mathbb{Z}_{2} \xrightarrow{\times d_{2}} \mathbb{Z}_{2} \xrightarrow{\times d_{3}} \cdots \rightarrow K_{*}(A).
\end{align}
In this case, however, the connecting maps will act like zero if they have even multiplicity and like 1 if they have odd multiplicity. Now let there are infinitely many even maps or in other words infinitely many zero maps and consider the sequence
\begin{align*}
G_{1} \xrightarrow{d_{1}}  G_{2} \xrightarrow{d_{2}} G_{3} \xrightarrow{d_{3}} \cdots
\end{align*}
where $G_{i}$ is a group for every $i$, in our case $G_{i}=\mathbb{Z}_{2}$ for every $i$, and the connecting map $d_{i}:G_{i} \rightarrow G_{i+1}$ is a group homomorphism for every $i$. Let the inductive limit of this sequence be $(G,\{ \mu_{i}\}_{i=1}^{\infty})$, where G is a group and $\mu_{i}:G_{i} \rightarrow G$ is a group homomorphism for every $i$. Suppose that there are infinitely many even maps in the sequence \eqref{seq}, which means that there are infinitely many zero maps in the sequence \eqref{seq}, in other words, for every $N \in \mathbb{N}$ there exists $M>N$ such that $d_{M}(x)=0$, where $x \in \mathbb{Z}_2$. Choose $ 1 \in G_{n}$, let us say 1 is in the finite stage related to $G_{n}$, but by mapping it into the next finite stage, and repeating (finitely many times is enough) we will hit one of the connecting maps that act like zero, i.e,. $\mu_{n,M+1}(1) =0$. Then applying the definition of the inductive limit $G=K_{1}(A)=K_{2}(A)=0$.
\end{proof}

\begin{corollary} \label{real UHF-classification}
Real UHF-algebras are classified by $K_{0}$.
\end{corollary}
\begin{proof}
Following Theorem \eqref{giordano} first we consider $K_{0}$, $K_{2}$ and $K_{4}$. Now let's begin with two real UHF-algebras $A$ and $B$, according to Lemma \eqref{Schroder} $K_{0}=K_{4}$ in the case of real UHF-algebras. If $K_{0}(A)=K_{0}(B)$, again according to Lemma \eqref{Schroder} $K_{2}(A)=K_{2}(B)$. indeed, let $\mathbb{Z}[\frac{1}{2}] \not\subset K_{0}(A)$ then since $K_{0}(A)=K_{0}(B)$ we have $K_{2}(A)=K_{2}(B)=\mathbb{Z}_{2}$. On the other hand, if $\mathbb{Z}[\frac{1}{2}] \subset K_{0}(A)$, according to Lemma \eqref{even map} $K_{2}(A)=K_{2}(B)=0$.

\end{proof}

\begin{theorem}\label{new theorem}
Let $A$ be a real UHF-algebra, then
\begin{align*}
K_{0}(A)=K_{0}(A \otimes_{\mathbb{R}} \mathbb{C})=K_{0}( A\otimes_{\mathbb{R}} \mathbb{H}).
\end{align*}
\begin{proof}
Since $A$ is a real UHF-algebra, then $A$, $A \otimes_{\mathbb{R}}\mathbb{C}$ and $A\otimes_{\mathbb{R}} \mathbb{H}$
are the inductive limit of the following sequences
\begin{align*}
M_{n_{1}}(\mathbb{R}) \rightarrow M_{k_{2}}(\mathbb{R}) \rightarrow M_{n_{3}}(\mathbb{R}) \rightarrow ... \rightarrow A,
\end{align*}
\begin{align*}
M_{n_{1}}(\mathbb{R}\otimes_{\mathbb{R}} \mathbb{C}) \rightarrow M_{n_{2}}(\mathbb{R}\otimes_{\mathbb{R}} \mathbb{C}) \rightarrow M_{n_{3}}(\mathbb{R}\otimes_{\mathbb{R}} \mathbb{C}) \rightarrow ... \rightarrow A\otimes_{\mathbb{R}} \mathbb{C}
\end{align*}
and
\begin{align*}
M_{n_{1}}(\mathbb{R}\otimes_{\mathbb{R}} \mathbb{H}) \rightarrow M_{n_{2}}(\mathbb{R}\otimes_{\mathbb{R}} \mathbb{H}) \rightarrow M_{n_{3}}(\mathbb{R}\otimes_{\mathbb{R}} \mathbb{H}) \rightarrow ... \rightarrow A\otimes_{\mathbb{R}} \mathbb{H}
\end{align*}
but $K_{0}(M_{k_{i}}(\mathbb{R}))=K_{0}(M_{k_{i}}(\mathbb{R}\otimes_{\mathbb{R}} \mathbb{C}))=K_{0}(M_{k_{i}}(\mathbb{R}\otimes_{\mathbb{R}} \mathbb{H})) =\mathbb{Z}$, and the connecting maps in these 3 sequences are the same we have the following commutative diagram

\begin{center}
\begin{tikzcd}
K_{0}(M_{n_{1}}(\mathbb{R})) \arrow[r,] \arrow[d, "\cong"] & K_{0}(M_{n_{2}}(\mathbb{R}))  \arrow[r] \arrow[d,"\cong"]& K_{0}(M_{n_{3}}(\mathbb{R}))  \arrow[r] \arrow[d,"\cong"] & ... \arrow[r] \arrow[d,"\cong"] & K_{0}(A) \arrow[d] \\
K_{0}(M_{n_{1}}(\mathbb{R})) \arrow[r,] \arrow[d, "\cong"] & K_{0}(M_{n_{2}}(\mathbb{R}))  \arrow[r] \arrow[d,"\cong"]& K_{0}(M_{n_{3}}(\mathbb{R}))  \arrow[r] \arrow[d,"\cong"] & ... \arrow[r] \arrow[d,"\cong"] & K_{0}(A \otimes_{\mathbb{R}} \mathbb{C}) \arrow[d] \\
K_{0}(M_{n_{1}}(\mathbb{H}))  \arrow[r] & K_{0}(M_{n_{2}}(\mathbb{H})) \arrow[r] & K_{0}(M_{n_{3}}(\mathbb{H})) \arrow[r] & ... \arrow[r]& K_{0}( A\otimes_{\mathbb{R}} \mathbb{H})  \\

\end{tikzcd}
\end{center}
which shows that $K_{0}(A)=K_{0}(A \otimes_{\mathbb{R}} \mathbb{C})=K_{0}(A \otimes_{\mathbb{R}} \mathbb{H})$.
\end{proof}
\end{theorem}

\begin{corollary}\label{real structe of UHF}
Every complex UHF-algebra has a  real structure which is a real UHF-algebra and it is unique up to isomorphism.
\end{corollary}
\begin{proof}
Let $A$ be a complex UHF-algebra then $K_{0}(A)$ is an additive sub-group of $\mathbb{Q}$ associated to a supernatural number, but there is a real UHF-algebra $B$ associated to that supernatural number, in other words, there is a real UHF-algebra $B$ for which we have $K_{0}(A)= K_{0}(B)$. Since $B$ is a real UHF-algebra it is the inductive limit of the sequence
\begin{align*}
M_{n_{1}}(\mathbb{R}) \rightarrow M_{k_{2}}(\mathbb{R}) \rightarrow M_{n_{3}}(\mathbb{R}) \rightarrow \cdots \to B,
\end{align*}
complexify this sequence we get
\begin{align*}
M_{n_{1}}(\mathbb{R})\otimes_{\mathbb{R}}\mathbb{C} \rightarrow M_{k_{2}}(\mathbb{R}) \otimes_{\mathbb{R}}\mathbb{C} \rightarrow M_{n_{3}}(\mathbb{R}) \otimes_{\mathbb{R}}\mathbb{C}\rightarrow \cdots\to B\otimes_{\mathbb{R}}\mathbb{C},
\end{align*}
Now we claim that $A \cong B\otimes_{\mathbb{R}}\mathbb{C}$ as a complex C*-algebra, indeed, as $B$ is a real UHF-algebra by Theorem \eqref{new theorem} we have $K_{0}(B)=K_{0}(B\otimes_{\mathbb{R}}\mathbb{C})$ and then since $K_{0}(A)=K_{0}(B)$ we have $K_{0}(A)=K_{0}(B\otimes_{\mathbb{R}}\mathbb{C})$, now  classification of complex UHF-algebra implies that $A \cong B\otimes_{\mathbb{R}}\mathbb{C}$, and this shows that for every complex UHF-algebra $A$ there is a real UHF-algebra whose complexification is $A$. Now we show that this real UHF-algebra  is unique up to isomorphism. Suppose that there is more than one real structures for a complex UHF-algebra $A$, we call them $B_{i}$. Apply Theorem \eqref{new theorem} we have $K_{0}(B_{i})=K_{0}(A_{\mathbb{C}})$ and then by the classification of real UHF-algebras ( Corollary \eqref{real UHF-classification}) we have $B_{i}$ are isomorphisms.

\end{proof}
\begin{lemma} \label{Ur tens Ur}
Let $U_{r}$ be the universal real UHF-algebra, then $U_{r} \otimes U_{r} \cong U_{r}. $
\end{lemma}
\begin{proof}
Applying Theorem \eqref{new theorem} we have $K_{0}(U_{r})=K_{0}(U_{r}\otimes_{\mathbb{R}} \mathbb{C} )$ and $K_{0}(U_{r} \otimes_{\mathbb{R}}U_{r})=K_{0}((U_{r}\otimes_{\mathbb{R}}U_{r})\otimes_{\mathbb{R}} \mathbb{C} )$, but since $K_{0}(U_{r}\otimes_{\mathbb{R}} \mathbb{C} )=K_{0}((U_{r}\otimes_{\mathbb{R}}U_{r})\otimes_{\mathbb{R}} \mathbb{C} )=\mathbb{Q}$, we have $$K_{0}(U_{r})=K_{0}(U_{r} \otimes_{\mathbb{R}}U_{r})=\mathbb{Q}.$$
Considering Lemma \eqref{Schroder} we have $KO_{0}(U_{r})=KO_{4}(U_{r})=\mathbb{Q}$ and $KO_{i}(U_{r})=0$ for $i=3, 5, 6, 7$, but on the other hand, since $U_{r}$ is a real UHF-algebra associated to the supernatural number which contains infinitely many even numbers, therefore by Lemma \eqref{even map} we have $KO_{1}(U_{r})=KO_{2}(U_{r})=0$. Now according to the Theorem \eqref{giordano} this kind of real UHF-algebras are classified only by $KO_{0}$. Thus by the classification since
\begin{align*}
K_{0}^{\text{CRT}}(U_{r} \otimes U_{r})= K_{0}^{\text{CRT}}(U_{r})= \mathbb{Q}
\end{align*}
we conclude that $U_{r} \otimes U_{r} \cong U_{r}. $
\end{proof}

\begin{definition}
Consider a collection of real C*-algebras and define the map $\chi_{2}^{R}(A):=m-n$ where $m$ and $n$ are defined as follows

\

$m$ $=$ $\mathbb{Q}$-dimension of $KO_{odd}(A \otimes U_{r})$ $-$ $\mathbb{Q}$-dimension of $KU_{odd}(A \otimes U_{r})$ $+$ $\mathbb{Q}$-dimension of $KT_{odd}(A \otimes U_{r})$,\

and

$n$ $=$ $\mathbb{Q}$-dimension of $KO_{even}(A \otimes U_{r})$ $-$ $\mathbb{Q}$-dimension of $KU_{even}(A \otimes U_{r})$ $+$ $\mathbb{Q}$-dimension of $KT_{even}(A \otimes U_{r})$.
\end{definition}
\begin{corollary}
Consider a collection of real C*-algebras, then $\chi_{2}^{R}$ defined above is additive.
\end{corollary}

\begin{proof}
Consider a short exact sequence of real C*-algebras
 \begin{align}\label{seq-r}
0 \longrightarrow A \longrightarrow C \longrightarrow B \longrightarrow 0
\end{align}

 and suppose that $U_{r}$ is the universal real UHF-algebra, therefore since  UHF-algebras are nuclear,  tensoring sequence \eqref{seq-r} by $U_{r}$ contributes to again a short exact sequence as follows
  \begin{align*}
0 \longrightarrow A \otimes U_{r} \longrightarrow C \otimes U_{r} \longrightarrow B\otimes U_{r}  \longrightarrow 0
\end{align*}
now according to the Theorem \eqref{cyclic} there exists a six term cyclic exact sequence
\[
\begin{tikzcd}
K^{\text{CRT}}_{even}(A\otimes U_{r}) \ar{r} & K^{\text{CRT}}_{even}(C \otimes U_{r}) \ar{r} & K^{\text{CRT}}_{even}(B \otimes U_{r})\ar{d} \\
K^{\text{CRT}}_{odd}(B \otimes U_{r}) \ar{u} & K^{\text{CRT}}_{odd}(C \otimes U_{r}) \ar{l} & K^{\text{CRT}}_{odd}(A \otimes U_{r}) \ar{l}
\end{tikzcd}
\]
but according to the Lemma \eqref{R.free ab}, $K_{*}^{\text{CRT}}(A_{i}\otimes U_{r})$ is a $\mathbb{Q}$-modulegroup for every real C*-algebra $A_{i}$, now applying Lemma \eqref{function} implies that $$\chi_{2}^{R}(C)=\chi_{2}^{R}(B)+\chi_{2}^{R}(A).$$.
\end{proof}

\begin{theorem}
Let $A$ or $B$ be in $N_{\mathbb{R}}$ with finitely generated $K^{\text{CRT}}$, then $\chi_{1}^{R}$ is multiplicative,i.e,. $\chi_{1}^{R}(A \otimes B)=\chi_{1}^{R}(A) \chi_{1}^{R}(B)$.
\end{theorem}

\begin{proof}
Let $A$ be in $N_{\mathbb{R}}$, then there is a short exact sequence
\begin{align*}
 0 \rightarrow K^{\text{CRT}}_{*}(A) \otimes_{\text{CRT}} K^{\text{CRT}}_{*}(B) \xrightarrow{\alpha}  K^{\text{CRT}}_{*}(A \otimes B) \xrightarrow{\beta} \mathrm{Tor}_{\text{CRT}}( K^{\text{CRT}}_{*}(A),K^{\text{CRT}}_{*}(B)) \rightarrow 0
 \end{align*}
 tensoring  by $\mathbb{Q}$ we have
  \begin{align*}
 (K^{\text{CRT}}_{*}(A) \otimes \mathbb{Q}) \otimes (K^{\text{CRT}}_{*}(B)\otimes \mathbb{Q}) \cong K^{\text{CRT}}_{*}(A \otimes B)\otimes \mathbb{Q}.
  \end{align*}
Let $K^{\text{CRT}}_{*}:= K^{\text{CRT}}_{even} \oplus K^{\text{CRT}}_{odd}$, then we have
 \begin{align*}
((K^{\text{CRT}}_{even}(A) \oplus K^{\text{CRT}}_{odd}(A) )\otimes \mathbb{Q}) \otimes ((K^{\text{CRT}}_{even}(B) \oplus K^{\text{CRT}}_{odd}(B) )\otimes \mathbb{Q}) \cong (K^{\text{CRT}}_{even}(A \otimes B) \oplus K^{\text{CRT}}_{odd}(A \otimes B) )\otimes \mathbb{Q}
  \end{align*}
which implies

  \begin{align*}
 (( K^{\text{CRT}}_{even}(A) \otimes K^{\text{CRT}}_{even}(B)) \otimes \mathbb{Q}) \oplus  (( K^{\text{CRT}}_{odd}(A) \otimes K^{\text{CRT}}_{odd}(B))  \otimes \mathbb{Q}) \cong K^{\text{CRT}}_{even}(A \otimes B)  \otimes \mathbb{Q}
  \end{align*}
  and

  \begin{align*}
 (( K^{\text{CRT}}_{even}(A) \otimes K^{\text{CRT}}_{odd}(B)) \otimes \mathbb{Q}) \oplus  (( K^{\text{CRT}}_{odd}(A) \otimes K^{\text{CRT}}_{even}(B)) \otimes
 \mathbb{Q}) \cong K^{\text{CRT}}_{odd}(A \otimes B) \otimes \mathbb{Q}
  \end{align*}
  This shows that the free part of $K^{\text{CRT}}_{even} (A \otimes B) $ is isomorphic as shown to a subspace of the free part of $K^{\text{CRT}}(A) \otimes K^{\text{CRT}}(B)$. Now by applying these two isomorphisms and definition of $\chi_{1}^{R}$ we have $\chi_{1}^{R}(A \otimes B)=\chi_{1}^{R}(A) \chi_{1}^{R}(B)$.
\end{proof}

\begin{theorem}
Let $A$ and $B$ be real C*-algebras and suppose that $A$ is in $N_{\mathbb{R}}$, then $\chi_{2}^{R}(A\tensor B)=\chi_{2}^{R}(A)\chi_{2}^{R}(B).$
\end{theorem}	

\begin{proof}
Since $N_{\mathbb{R}}$ is closed under tensor product, by hypothesis $A \otimes U_{r}$ is in $N_{\mathbb{R}}$, then using Lemma \eqref{Ur tens Ur} contributes a short exact sequences

 \begin{multline*}0 \rightarrow K^{\text{CRT}}_{even}(A \otimes U_{r}) \otimes K^{\text{CRT}}_{even}(B \otimes U_{r}) \oplus K^{\text{CRT}}_{odd}(A \otimes U_{r}) \otimes ^{\text{CRT}}_{odd}(B \otimes U_{r}) \rightarrow K^{\text{CRT}}_{even}((A \otimes B)\otimes U_{r}) \rightarrow\\ \mathrm{Tor}(K^{\text{CRT}}_{even}(A\otimes U_{r}),K^{\text{CRT}}_{odd}(B \otimes U_{r})) \oplus \mathrm{Tor}(K^{\text{CRT}}_{odd}(A \otimes U_{r}),K^{\text{CRT}}_{even} (B\otimes U_{r})) \rightarrow 0
 \end{multline*}
 and
 \begin{multline*}0 \rightarrow K^{\text{CRT}}_{even}(A \otimes U_{r}) \otimes K^{\text{CRT}}_{odd}(B \otimes U_{r}) \oplus K^{\text{CRT}}_{odd}(A \otimes U_{r}) \otimes K^{\text{CRT}}_{even}(B \otimes U_{r}) \rightarrow K^{\text{CRT}}_{odd}((A \otimes B)\otimes U_{r}) \rightarrow\\ \mathrm{Tor}(K^{\text{CRT}}_{even}(A\otimes U_{r}),K^{\text{CRT}}_{even}(B \otimes U_{r})) \oplus \mathrm{Tor}(K^{\text{CRT}}_{odd}(A \otimes U_{r}),K^{\text{CRT}}_{odd}(B \otimes U_{r})) \rightarrow 0
 \end{multline*}
 but $K_*^{\text{CRT}}(A_{i} \otimes U_{r})$ is torsion free for every real C*-algebra $A_{i}$, thus we have two isomorphisms
 $$K^{\text{CRT}}_{even} (A\otimes U) \otimes K^{\text{CRT}}_{even} (B\otimes U) \oplus K^{\text{CRT}}_{odd}(A\otimes U)\otimes K^{\text{CRT}}_{odd}(B\otimes U) = K^{\text{CRT}}_{even} ((A \otimes B)\otimes U),$$
 and
 $$K^{\text{CRT}}_{even} (A\otimes U) \otimes K^{\text{CRT}}_{odd}(B\otimes U \oplus (K^{\text{CRT}}_{odd}(A\otimes U)\otimes K^{\text{CRT}}_{even} (B\otimes U) = K^{\text{CRT}}_{odd}((A \otimes B)\otimes U)$$
as before again since isomorphisms preserve the number of free generators so by applying these two isomorphisms and definition of $\chi_{2}$ we get $\chi_{2}^{R}(A\tensor B)=\chi_{2}^{R}(A)\chi_{2}^{R}(B).$
\end{proof}

\begin{lemma} \label{isomorphism}
Consider the real C*-algebras $T=\{f \in C([0,1], \mathbb{C}) : f(0)=\overline{f(1)}\}$ and $C(T,n_1) = \{ g \in C(S^{1}, \mathbb{C}): g(-x)=\overline{g(x)} \ \text{for every}  \ x \in S^{1}\}$, where $S^{1}$ is the unit circle. Then $C(T,n_1)$ is isomorphic to $T$ as a real C*-algebra, and the their complexifications are isomorphic to $C(S^1)$ as a complex $C^{*}$-algebra.
\end{lemma}
\begin{proof}
Consider $C(S^{1})$ the set of all continuous functions $f:S^{1} \rightarrow \mathbb{C}$, using the antilinear involution

\begin{align*}
\alpha : C(S^{1}) \ni f(x) \mapsto  \overline{f(-x) }\in C(S^{1})
\end{align*}
where $x \in S^{1}$, we get the real subalgebra associated to $\alpha$  denoted by
\begin{align*}
C(T,n_1) &= \{ g \in C(S^{1}, \mathbb{C}): \alpha(g(x))= g(x)\} \\
& = \{ g \in C(S^{1}, \mathbb{C}): \overline{g(-x)}= g(x)\}.
\end{align*}
In a nutshell, we recovered the real C*-algebra $C(T,n_1)$, from the complex C*-algebra $C(S^{1})$, and thus $C(T,n_1) \otimes_{\mathbb{R}} \mathbb{C} = C(S^{1})$ as a complex C*-algebra, which means that the complexification of this real C*-algebra is the complex C*-algebra $C(S^{1})$.

Now let us to show that $C(T,n_1)$ is isomorphic to $T$. Suppose that $\varphi: C(T,n_1) \rightarrow T $ for which the domain is restricted to $[\pi,0]$ the lower half of the unit circle. But by the definition of $C(T,n_1)$ we have $g(\pi)=\overline{g(0)}$, for every $g$ in the lower half of the unit circle. And this is nothing but $T$. On the other hand, define the map $\psi : T \rightarrow C(T,n_1)$ by
\begin{align*}
\psi(f(x)) = g(x)  := \left\{ \begin{array}{cc}
                f(x) & x \in [0, \pi] \\
                                               \\
               \overline{f(-x)} &  x \in [\pi,0] \\
                \end{array} \right.
\end{align*}
  therefore $C(T,n_1) =T$, and $C(S^{1}) = C(T,n_{1}) \otimes_{\mathbb{R}} \mathbb{C}=T \otimes_{\mathbb{R}} \mathbb{C}$.
\end{proof}

\begin{theorem}\label{theorem 8.10}
Let $A$ be a complex C*-algebra with a finitely generated K-theory and it is in the UCT class, then $\chi_{\mathbb{R}}(A)=2 \chi_{\mathbb{C}}(A)$, where $\chi_{\mathbb{R}}$ is the character for Grothendieck ring of the real C*-algebra , and  $\chi_{\mathbb{C}}$ is the character for Grothendieck ring of complex C*-algebra.
\end{theorem}

\begin{proof}
By Lemma \eqref{isomorphism} we have
\begin{align*}
C(S^{1}) = C(T,n_{1}) \otimes_{\mathbb{R}} \mathbb{C}=T \otimes_{\mathbb{R}} \mathbb{C}
\end{align*}
therefore
 \begin{align*}
C(S^{1}) \otimes_{\mathbb{C}} A= (T \otimes_{\mathbb{R}} \mathbb{C}) \otimes_{\mathbb{C}} A
\end{align*}
which implies
\begin{align*}
C(S^{1},A) = T \otimes_{\mathbb{R}} A
\end{align*}
but
\begin{align*}
K_{n}(C(S^{1},A))=K_{n}(A) \oplus K_{n+1}(A),
\end{align*}
which implies
\begin{align} \label{T tensor A}
K_{n}(T \otimes_{\mathbb{R}} A )=K_{n}(A) \oplus K_{n+1}(A).
\end{align}

Recall that $\chi_{\mathbb{C}}(A)=m-n$, where $m$ is the number of free generators of $K_{1}(A)$, and $n$ is the number of free generators of $K_{0}(A)$. Since $A$ is a complex $C*-algebra, KO_{n}(A)=K_{n}(A)$, then using Definition \eqref{real character} we have:
\begin{quotation}
the number of free generators of $KO_{odd}(A)$ - the number of free generators of $KO_{even}(A)$ = $4(m-n)$.
\end{quotation}
On the other hand by \eqref{T tensor A} we have:
\begin{quotation}
the number of free generators of $KT_{odd}(A)$ - the number of free generators of $KT_{even}(A)$ = 0.
\end{quotation}
We recall that $KU_{n}(A)=K_{n}(A \otimes_{\mathbb{R}} C)$. Thus
\begin{align*}
KU_{n}(A)=K_{n}(A \otimes_{\mathbb{R}}\mathbb{C}) &=K_{n}((A \otimes_\mathbb{C} \mathbb{C})\otimes_{\mathbb{R}} \mathbb{C}) \\
&=K_{n}(A \otimes_\mathbb{C} (\mathbb{C}\otimes_{\mathbb{R}} \mathbb{C})) \\
&= K_{n}(A \otimes_\mathbb{C} (\mathbb{C} \oplus \mathbb{C})) \\
&=2K_{n}(A \otimes_\mathbb{C} \mathbb{C})\\
&=2K_{n}(A)
\end{align*}
Now we have
\begin{align*}
\chi_{\mathbb{R}}(A)= 4(m-n)-2(m-n)=2 \chi_{\mathbb{C}}(A).
\end{align*}
\end{proof}

\begin{lemma} \label{lemma 8.11}
Let $A$ be a complex $C^{*}$-algebra for which there exists a real structure $A_{Re}$, then $\chi_{\mathbb{R}}(A)=2 \chi_{\mathbb{R}}(A_{Re})$.
\end{lemma}

\begin{proof}
$\chi_{\mathbb{R}}(A)= \chi_{\mathbb{R}}(A_{Re} \otimes_{\mathbb{R}} \mathbb{C})=\chi_{\mathbb{R}}(A_{Re}) \chi_{\mathbb{R}}(\mathbb{C})= 2 \chi_{\mathbb{R}}(A_{Re})$.
\end{proof}

\begin{corollary}
Let $A$ be a complex $C^{*}$-algebra for which there exists a real structure $A_{Re}$, then $\chi_{\mathbb{C}}(A)=\chi_{\mathbb{R}}(A_{Re})$.
\end{corollary}
\begin{proof}
 Applying Theorem \eqref{theorem 8.10} and Lemma \eqref{lemma 8.11}.
\end{proof}

\begin{example}
Consider $\mathbb{C}$ as a complex $C^{*}-algebra$ then using table 1 and table 2 in (Boersema) we have
$$\chi_{\mathbb{C}}(\mathbb{C})=\chi_{\mathbb{R}}(\mathbb{R})=1.$$
\end{example}

\section{Characters on Grothendieck rings of quaternion C*-algebras}
In this section we first define $\mathbb{H}$-C*-algebras and then introduce a candidate character on their Grothendieck rings.

$A$ is a $\mathbb{H}$-vector space if $A$ is a vector space with respect to the right and left multiplication by $\mathbb{H}$. Then we say that $A$ is a $\mathbb{H}$-algebra is the following properties hold:
 \begin{align*}
\lambda(ab)=(\lambda a)b, \ (a \lambda)b=a(\lambda b), \ and \ (ab)\lambda=a(b\lambda)
 \end{align*}
for every $a, b \in A$ and $\lambda \in \mathbb{H}$.
Moreover, an involution on $A$ is a map $*:A \rightarrow A$ which satisfying:
\begin{align*}
(\lambda a )^{*}=a^{*} \lambda^{*} \ and \ ( a \lambda  )^{*}= \lambda^{*} a^{*}.
\end{align*}
Now suppose that $A$ is an $\mathbb{H}$-involutive algebra with a $\mathbb{H}$-norm $ \|.\|$, then we sat that $A$ is a $\mathbb{H}$-C*-algebra if it is complete under $ \|.\|$ and satisfies the following properties:
\begin{align*}
\|ab\| \leqslant \|a\| \|b\|  \ and  \ \|a^{*}a\|=\|a\|^{2}.
\end{align*}

Let $A$ be a quaternion $C^{*}$-algebra then the real part of $A$ can be recovered by
  $$A_{Re}=\{ a \in A : \lambda a = a \lambda \ for \ every \ \lambda \in \mathbb{H} \}.$$

Now we consider the Grothendieck rings of $\mathbb{H}$-$C^{*}$-algebras. In order to introduce a candidate character on these kinds of Grothendieck rings, if we want to go through the same process the same as complex and real case then we need use quaternion  K\"unneth theorem to show that the character is multiplicative. Such a sequence is in effect provided by Boersema’s Kunneth-type theorem for the united K-theory of $\mathbb{H}$-$C^{*}$-algebras.

\begin{theorem}(Wood-Karoubi)
For any real C*-algebra $A$ the following isomorphisms hold:
\begin{align*}
K_{n}(A) \cong K_{n+8}(A) \ and \ K_{n}(A) \cong K_{n+4}(A \otimes_{\mathbb{R}} \mathbb{H}).
\end{align*}
\end{theorem}
\begin{lemma}\label{complex-H}
$\mathbb{H} \otimes_{\mathbb{R}} \mathbb{C} \cong M_{2}(\mathbb{C})$ as a complex $C^{*}$-algebra.
\end{lemma}
\begin{proof}
Consider the map $\phi: \mathbb{H} \otimes_{\mathbb{R}} \mathbb{C} \rightarrow M_{2}(\mathbb{C})$ defined by
$$ \phi(u\otimes_{\mathbb{R}}1)=\begin{psmallmatrix} \bar{z} & \bar{w}\\ -w & z\end{psmallmatrix}$$
where $u \in \mathbb{H}$ and $z,w \in \mathbb{C}$, then $\phi$ is an isomorphism.
\end{proof}

\begin{theorem}
Let $A$ be a $\mathbb{H}$-$C^{*}$-algebra, then $\chi_{\mathbb{R}}(A)=\chi_{\mathbb{R}}(A_{Re})$.
\end{theorem}

\begin{proof}
Since $A=A_{Re} \otimes_{\mathbb{R}} \mathbb{H}$, we have $\chi_{\mathbb{R}}(A)=\chi_{\mathbb{R}}(A_{Re} \otimes_{\mathbb{R}} \mathbb{H})=\chi_{\mathbb{R}}(A_{Re}) \chi_{\mathbb{R}}(\mathbb{H})$. Now we calculate $\chi_{\mathbb{R}}(\mathbb{H})$, which means that we need to calculate $K^{\text{CRT}\text{CRT}}(\mathbb{H})$. The real K-theory $KO_{n}(\mathbb{H})=K_{n}(\mathbb{H})$ is given as below
\begin{table}[!h]
\begin{center}
\begin{tabular}{l l l l l l l l l }
   \label{p1}&0   & 1&2&3 &4 &5 &6&7 \\
    $K_{n}(\mathbb{H})$ &$\mathbb{Z}$ & 0& 0& 0 & $\mathbb{Z}$ &$\mathbb{Z}/2$ & $\mathbb{Z}/2$& 0\\
\end{tabular}
\end{center}
\end{table}

$KU_{n}(\mathbb{H})=K_{n} ( \mathbb{H}\otimes_{\mathbb{R}} \mathbb{C})$, but by Lemma \eqref{complex-H} $K_{n}(\mathbb{H}\otimes_{\mathbb{R}} \mathbb{C}))=K_{n}(M_{2}(\mathbb{C}))$, therefore we have

\begin{table}[!h]
\begin{center}
\begin{tabular}{l l l l l l l l l }
   \label{p1}&0   & 1&2&3 &4 &5 &6&7 \\
    $KU_{n}(\mathbb{H})$ &$\mathbb{Z}$ & 0& $\mathbb{Z}$& 0 & $\mathbb{Z}$ &0 & $\mathbb{Z}$& 0\\
\end{tabular}
\end{center}
\end{table}

In order to calculate $KT_{n}(\mathbb{H})$, we apply the sequence (Boersema)
\begin{align*}
... \rightarrow KU_{n+1}(A) \xrightarrow{\gamma} KT_{n}(A) \xrightarrow{\zeta} KU_{n}(A) \xrightarrow{1-\psi_{U}} KU_{n}(A) \rightarrow  ... .
\end{align*}
where $A$ is a real C*-algebra, now take $A=\mathbb{R}$ and then using the Wood-Karoubi theorem we obtain
\begin{table}[!h]
\begin{center}
\begin{tabular}{l l l l l l l l l }
   \label{p1}&0   & 1&2&3 &4 &5 &6&7 \\
    $KT_{n}(\mathbb{H})$ &$\mathbb{Z}$ & $\mathbb{Z}/2$& 0& $\mathbb{Z}$ & $\mathbb{Z}$ &$\mathbb{Z}/2$& 0& $\mathbb{Z}$\\
\end{tabular}
\end{center}
\end{table}

using above $K_{n}(\mathbb{H})$, $KU_{n}(\mathbb{H})$ and $KT_{n}(\mathbb{H})$, then $\chi_{\mathbb{R}}(\mathbb{H})=1$ which implies that $\chi_{\mathbb{R}}(A)=\chi_{\mathbb{R}}(A_{Re})$.
\end{proof}
Thus, we have defined and computed characters on Grothedieck rings of C*-algebras over $\mathbb{R}$, $\mathbb{C}$, and $\mathbb{H}$. Grothendieck rings of C*-algebras have been studied very little, but there is a more developed theory in the setting of classical algebraic geometry, and in that case, the existence of characters was the key step towards showing that the classic Grothendieck ring of varieties is not the trivial ring $\{0\}.$ In our more complicated case, we have succeeded in showing that some  Grothedieck rings of C*-algebras are trivial (Theorem 2.1) but that in some other cases, there exist nontrivial Grothendieck rings of C*-algebras over $\mathbb{R}$, $\mathbb{C}$, and $\mathbb{H}$.

\centerline{{\huge {\rotatebox[]{270}{\textdagger}}\!{\rotatebox[]{90}{\textdagger}}}}
\end{document}